\documentclass[11pt,a4paper]{article}

\newcommand{\final}{True}
\usepackage[utf8]{inputenc}
\usepackage[T1]{fontenc}
\usepackage{amsmath, amsthm, amsopn, amsfonts, amssymb, dsfont, color, mathrsfs, dsfont, authblk, graphicx, comment, xcolor, a4, a4wide, xpatch, lmodern, enumitem, stmaryrd, ifthen}

\definecolor{blue}{rgb}{0,0,.6}
\usepackage[%
	pdftitle={The field-road diffusion model: fundamental solution and asymptotic behavior},
	pdfauthor={Matthieu Alfaro, Romain Ducasse, Samuel Tréton},%
	pdfkeywords={field-road model, diffusion, exchange boundary conditions, fundamental solution, decay rate},
	bookmarks=true,
	bookmarksopen=true,
	colorlinks=true,
	urlcolor=black,
	linkcolor=blue,
	citecolor=blue
           ]{hyperref}

\newcommand{\crochets}[1]{\left[#1\right]}		%Les crochets
\newcommand{\Parentheses}[1]{\left(#1\right)}		%Les Parentheses
\newcommand{\parentheses}[1]{(#1)}			%Les parentheses
\newcommand{\ensemble}[1]{\left\lbrace#1\right\rbrace }	%Les accolades

\newcommand{\intervalleff}[2]{\left[#1,#2\right]}

\newcommand{\intervalleoo}[2]{\left(#1,#2\right)}
\newcommand{\intervalleof}[2]{\left(#1,#2\right]}

\newcommand{\MAT}[1]{\ifthenelse{\equal{\final}{False}}{\textcolor{red}{#1}}{#1}}

\let\sqrtTEMP\sqrt
\def\sqrt#1{\sqrtTEMP{#1\,}}

\renewcommand{\=}{:\hspace{0.5mm}=}	% := (pour les définitions)

\newcommand{\Fb}{\mathcal{F}} %Transformée de Fourier
\newcommand{\Lb}{\mathcal{L}} %Transformée de Laplace
\renewcommand{\frown}[1]{\hspace{-0.379mm}{\stackrel{\includegraphics[scale=1]{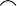}}{#1}}}
\renewcommand{\hat}[1]{\hspace{-0.379mm}{\stackrel{\includegraphics[scale=1]{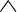}}{#1}}}
\newcommand{\frownhat}[1]{\hspace{-0.379mm}{\stackrel{\includegraphics[scale=1]{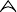}}{#1}}}

\newcommand{\point}{\includegraphics[scale=1]{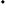}}

\newcommand{\indicatrice}[1]{\mathds{1}_{#1}}	%Fonction indicatrice (besoin du Package Dsfont)

\newcommand{\lessim}{\lesssim}

\newcommand{\GRd}{G_{{\!\scriptscriptstyle{R}}}^{{\scriptscriptstyle{d}}}}
\newcommand{\GRD}{G_{{\!\scriptscriptstyle{R}}}^{{\scriptscriptstyle{D}}}}

\newcommand{\HalfaONE}{H_{\alfa}^{{\scriptscriptstyle{\parentheses{1}}}}}
\newcommand{\epsR}{\varepsilon_{{\!\scriptscriptstyle{R}}}}
\newcommand{\epsI}{\varepsilon_{{\!\scriptscriptstyle{I}}}}

\newcommand{\CvzeroSHARP}{\,\raisebox{-1.290mm}{\includegraphics[scale=1]{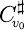}}\,}
\newcommand{\Cvzero}{\,\raisebox{-1.290mm}{\includegraphics[scale=1]{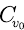}}\,}

\newcommand{\TCvzero}{\,\raisebox{-1.290mm}{\includegraphics[scale=1]{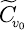}}\,}

\newcommand{\Cvuzero}{\,\raisebox{-1.290mm}{\includegraphics[scale=1]{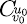}}\,}

\newcommand{\?}{\hspace{0.1mm}}	%Pour définir \, mais en plus fin

\newcommand{\verti}[1]{\left\vert #1 \right\vert}
\newcommand{\vertii}[2]{\Vert #1 \Vert_{#2}}

\renewcommand\Re{\operatorname{\mathfrak{Re}}}
\renewcommand{\Im}{\operatorname{\mathfrak{Im}}}
\newcommand{\realpart}[1]{\Re\parentheses{#1}}
\newcommand{\impart}[1]{\Im\parentheses{#1}}
\newcommand{\Realpart}[1]{\Re\Parentheses{#1}}
\newcommand{\Impart}[1]{\Im\Parentheses{#1}}

\newcommand{\Sb}{\mathcal{S}}
\newcommand{\Eb}{\mathcal{E}}

\makeatletter
\AtBeginDocument{\xpatchcmd{\@thm}{\thm@headpunct{.}}{\thm@headpunct{}}{}{}}
\makeatother

\newcommand{\alfa}{\theta}

\newcounter{FIGURE}
\renewcommand{\theFIGURE}{\Roman{FIGURE}}
\newcounter{HEHalfSpace}
\renewcommand{\theHEHalfSpace}{\roman{HEHalfSpace}}
\newcounter{AUn}
\renewcommand{\theAUn}{\roman{AUn}}
\newcounter{ADeux}
\renewcommand{\theADeux}{\roman{ADeux}}
\newcounter{ALargeDelta}

\newcounter{ASimpleRoots}
\renewcommand{\theASimpleRoots}{\roman{ASimpleRoots}}
\newcounter{ATripleRoots}
\renewcommand{\theATripleRoots}{\roman{ATripleRoots}}
\newcounter{ATwoDoubleRoots}
\renewcommand{\theATwoDoubleRoots}{\roman{ATwoDoubleRoots}}
\newcounter{AOneDoubleRoots}
\renewcommand{\theAOneDoubleRoots}{\roman{AOneDoubleRoots}}
\makeatletter	%(Mettre à gauche)
\@mparswitchfalse
\reversemarginpar
\makeatother

\def\R{\mathbb R}

\newtheorem{theorem}{\textbf{Theorem}}[section]

\newtheorem{lemma}[theorem]{\textbf{Lemma}}

\newtheorem{corollary}[theorem]{\textbf{Corollary}}

\newtheorem{remark}[theorem]{\textbf{Remark}}

\numberwithin{equation}{section}

\newcommand{\Erfc}{\textsl{Erfc}\,}

\title{The field-road diffusion model: fundamental solution and asymptotic behavior} %(Si modification, changer aussi dans les métadonnées)
\date{}	%Uncomment to not show the date

\author[1]{Matthieu {\sc Alfaro}}
\author[2]{Romain {\sc Ducasse}}
\author[1]{Samuel {\sc Tr\'eton}}
\affil[1]{Universit\'e de Rouen Normandie, CNRS, Laboratoire de Math\'ematiques {Rapha\"el Salem} (LMRS), Saint-Étienne-du-Rouvray, France.}
\affil[2]{Université Paris Cité and Sorbonne Université, CNRS, Laboratoire
{Jacques-Louis Lions} (LJLL), F-75006 Paris, France.}

\begin{document}

\maketitle

\begin{abstract} 
We consider the {\it linear} field-road system, a model for {\it fast diffusion channels} in population dynamics and ecology. This system takes the form of a system of PDEs set on domains of different dimensions, with exchange boundary conditions.
Despite the intricate geometry of the problem, we provide an explicit expression for its fundamental solution and for the solution to the associated Cauchy problem. The main tool is a Fourier (on the road variable)/Laplace (on time) transform. In addition, we derive estimates for the decay rate of the $L^\infty$ norm of these solutions.
\end{abstract}
\vspace{3pt}


\vspace{3pt}


%%%%%%%%%%%%%%%%%%%%%%%%%%%%%%%%%%%%%%%%%%%%%%%%%%%%%%%%%%%%%%%%%%%%%%%%%%
%%%%%%%%%%%%%%%%%%%%%%%%%%%%%%%%%%%%%%%%%%%%%%%%%%%%%%%%%%%%%%%%%%%%%%%%%%
%%%%%%%%%%%%%%%%%%%%%%%%%%%%%%%%%%%%%%%%%%%%%%%%%%%%%%%%%%%%%%%%%%%%%%%%%%

\section{Introduction}

In this work we consider the solutions $\Parentheses{v=v(t,x,y), u=u(t,x)}$ to the {\it linear} field-road model
\label{s:intro}
\begin{equation}\label{syst}
\left\{
\begin{array}{ll}
	\partial_{t} v = d \Delta v,
&\quad t>0, \; x \in \R^{N-1}, \; y>0, \vspace{4pt}\\
	- d \partial_{y} v|_{y=0} = \mu u - \nu v|_{y=0},
&\quad t>0, \; x \in \R^{N-1}, \vspace{4pt}\\
	\partial_{t} u = D \Delta u + \nu v|_{y=0} - \mu u,
&\quad t>0, \; x \in \R^{N-1}.
\end{array}
\right.
\end{equation}
Here $N\geq 2$ and $d$, $D$, $\mu$ and $\nu$ are positive constants. This system is actually the linear (purely diffusive) part of the \emph{field-road} model, introduced by Berestycki, Roquejoffre and Rossi \cite{Ber-Roq-Ros-13-1} to describe the spread of invasive species in presence of networks with fast propagation. We review the original system later in this section.

\medskip

Let us mention that phenomena of spatial spread are highly relevant to the understanding of biological invasions, spreads of emergent diseases, as well as spatial shifts in distributions in the context of climate change. There is a wide literature dedicated to these topics, let us refer for instance to \cite{Murray1, Murray2} and the reference therein for more details. In the recent years, there has been a growing recognition of the importance of {\it fast diffusion channels} on biological propagations: for instance, an accidental transportation via human activities of some individuals towards northern and eastern France may be the cause of accelerated propagation of the pine
processionary moth \cite{Rob-et-al-12}. Also, in Canada, some GPS data revealed that wolves travel faster along seismic lines (i.e. narrow strips cleared for energy exploration), thus increasing their chances to meet a prey \cite{MacKen-et-al-12}. It is also acknowledged that fast diffusion channels (roads, airlines, etc.) play a central role in the propagation of epidemics. As is well known, the spread of the black plague, which killed about a third of the European population in the 14th century, was favoured by the trade routes, especially the Silk Road, see  \cite{PesteNoireRoutesSoie}. More recently, some evidences of the the radiation of the COVID epidemic along highways and transportation infrastructures were found \cite{Gat-et-al-20}.

In order to capture the phenomena induced by such fast diffusion channels, the so-called field-road reaction diffusion system 
\begin{equation}\label{syst-non-lin}
\left\{
\begin{array}{ll}
	\partial_{t} v = d \Delta v + f(v),
&\quad t>0, \; x \in \R^{N-1}, \; y>0, \vspace{4pt}\\
	- d \partial_{y} v|_{y=0} = \mu u - \nu v|_{y=0},
&\quad t>0, \; x \in \R^{N-1}, \vspace{4pt}\\
	\partial_{t} u = D \Delta u + \nu v|_{y=0} - \mu u,
&\quad t>0, \; x \in \R^{N-1},
\end{array}
\right.
\end{equation}
was proposed by Berestycki, Roquejoffre and Rossi \cite{Ber-Roq-Ros-13-1}. The mathematical problem then amounts to describing survival and propagation in a non-standard physical space: the geographical domain consists in the half-space (the \lq\lq field'') $x\in \R^{N-1}$, $y>0$, bordered by the hyperplane (the \lq\lq road'') $x\in \R^{N-1}$, $y=0$. In the field, individuals diffuse with coefficient $d>0$ and their density is given  by $v=v(t,x,y)$. In particular $\Delta v$ has to be understood as $\Delta_x v+\partial_{yy}v$. On the road, individuals typically diffuse faster ($D>d$) and their density is given by $u=u(t,x)$. In particular $\Delta u$ has to be understood as $\Delta_x u$. The exchanges of population between the road and the field are described by the second equation in system \eqref{syst-non-lin}, where $\mu>0$ and $\nu >0$. These Robin type boundary conditions link the field and the road equations and, in some sense, are the core of the model. 

In a series of works \cite{Ber-Roq-Ros-13-1, Ber-Roq-Ros-13-2, Ber-Roq-Ros-tw, Ber-Roq-Ros-shape},  Berestycki, Roquejoffre and Rossi studied the field-road system with $N=2$ and $f$ a Fisher-KPP nonlinearity. They shed light on an {\it acceleration phenomenon}: when $D>2d$, the road  enhances the global diffusion and  the spreading speed exceeds the standard Fisher-KPP invasion speed. Since then, many generalizations or related problems have been studied. Field-road models with (periodic) heterogeneities are considered in \cite{Gil-Mon-Zho-15, Zha-21, Aff-22}. Situations where the field is not a half-space but a more general domain are studied in \cite{Tel-16, Ros-Tel-Val-17, Duc-18, Bog-Gil-Tel-21}. The setting where the Laplace operators are replaced with non-local operators, accounting for more complex diffusion processes, is tackled in \cite{AC1, AC2}. The papers \cite{BDR1, Ber-Duc-Ros-20} consider the interaction between fast-diffusion networks and ecological phenomena such as climate change. In \cite{PauthierLongRangeExchanges1, PauthierLongRangeExchanges2, PauthierLongRangeExchanges3}, the author introduces \textit{long range exchanges} so that the road may catch individuals from anywhere in the field and \textit{vice versa}.

\medskip

Despite these results on \eqref{syst-non-lin} and its variations,  it turned out that the {\it linear} field-road system --- obtained by letting $f\equiv 0$ in \eqref{syst-non-lin} --- was not fully understood: there were no complete expressions for the fundamental solution, and the decay rate of the $L^{\infty}$ norm of the solutions to \eqref{syst} as $t\to +\infty$ was not known.

In the present paper, we thus consider \eqref{syst} as a starting point and intend at filling this gap. It turns out that we can reach an explicit expression for both the fundamental solution and the solution to the associated Cauchy problem. To do so, we perform a \lq\lq Fourier in $x$/Laplace in $t$'' transform, which seems to be the good way to understand the intricate exchange boundary conditions. These explicit expressions  enable, in particular, to provide a sharp (possibly up to a logarithmic term) decay rate of the $L^\infty$ norm of the solution.

Such a linear rate estimate is a key stone information to tackle difficult nonlinear issues. For instance, it is expected, see \cite{Fuj-66} or \cite{Alf-Kav-21}, that such a decay rate determines the so-called {\it $\text{Fujita}$ exponent} $p_F$ for the system \eqref{syst-non-lin} when $f(v)=v^{1+p}$: this exponent separates \lq\lq systematic blow-up'' (when $0<p\leq p_F)$ from \lq\lq possible extinction'' (when $p>p_F$). Also, understanding the Fujita blow-up phenomena when $f(v)=v^{1+p}$ typically enables, see \cite{Alf-fuj-17}, to solve the issue of \lq\lq extinction \textit{vs} propagation'' in the population dynamics model \eqref{syst-non-lin} with $f(v)=v^{1+p}(1-v)$. Such a nonlinearity  serves a model for the so-called {\it Allee effect} \cite{Allee}, roughly meaning that the {\it per capita} growth rate of the population is not maximal at small density. We plan to address such issues in a future work.

%%%%%%%%%%%%%%%%%%%%%%%%%%%%%%%%%%%%%%%%%%%%%%%%%%%%%%%%%%%%%%%%%%%%%%%%%%
%%%%%%%%%%%%%%%%%%%%%%%%%%%%%%%%%%%%%%%%%%%%%%%%%%%%%%%%%%%%%%%%%%%%%%%%%%
%%%%%%%%%%%%%%%%%%%%%%%%%%%%%%%%%%%%%%%%%%%%%%%%%%%%%%%%%%%%%%%%%%%%%%%%%%

\section{Main results}\label{s:results}

Let $N \geq 2$ be an integer. A generic point $X\in \R^{N}$ will be written $X=(x,y)$ with $x\in \R^{N-1}$ and $y\in \R$. For the \lq\lq variable of integration'', we reserve the notation ${Z=(z,\omega)\in \R^N}$, with $z\in \R^{N-1}$ and $\omega \in \R$. For the \lq\lq $x$-Fourier variable'', we use the notation $\xi \in \R^{N-1}$, while, for the \lq\lq $t$-Laplace variable'', we reserve the notation $s>0$, see Appendix \ref{ss:appendix_Transforms}. We denote the upper half-space (the \lq\lq field'') $\R^{N}_{+} \= \R^{N-1} \times (0,+\infty)$, its boundary being the hyperplane (the \lq\lq road'') ${\partial\R^N_{+} = \R^{N-1}\times\{0\}}$.\medskip

For given $\mu,\nu>0$ and $D>d>0$ (we refer to Remark \ref{rem:D-egal-d} for the case $D\leq d$), we thus consider the linear field-road problem \eqref{syst} supplemented with the initial datum
\begin{equation}\label{initial_data}
\left\{
\begin{array}{lll}
	v|_{t=0} = v_{0},
	&\quad X \in \mathbb{R}^{N}_{+},
	&\quad v_{0}\in L^{\infty}\parentheses{\mathbb{R}^{N}_{+}}, \vspace{4pt}\\
		u|_{t=0} = u_{0},
	&\quad x \in \R^{N-1},
	&\quad u_{0}\in  L^{\infty}\parentheses{\mathbb{R}^{N-1}}. 
\end{array}
\right.
\end{equation}
Our first main contribution is to provide an explicit expression of the solution to this Cauchy problem. 

\begin{theorem}[Solution to the \textit{linear} field-road Cauchy problem]\label{th:heat-eq-FR-space-SOLUTION} Assume $D>d$. Then the unique bounded solution to the Cauchy problem \eqref{syst}\?--\,\eqref{initial_data} is
\begin{alignat}{3}
\nonumber
&v\parentheses{t,X} &\; = &\;
	V\parentheses{t,X}
	+
	\frac{\mu}{\sqrt{d}} \int_{\mathbb{R}^{N-1}}^{}
		\Lambda\parentheses{t,z,y} \;
		u_{0}\parentheses{x-z} \;
	dz \\
	\label{v-sol-FR} 
	&&&\hspace{26.593mm}+ \frac{\mu \, \nu}{\sqrt{d}} \int_{0}^{t}
		\int_{\mathbb{R}^{N-1}}^{}
			\Lambda\parentheses{s,z,y} \;
			V|_{y=0}\parentheses{t-s,x-z} \;
		dz \,
	ds,\\[10pt]
\label{u-sol-FR}
&u\parentheses{t,x} &\; = &\
	e^{-\mu t} U\parentheses{t,x}
	+
	\nu \int_{0}^{t}
		e^{-\mu\parentheses{t-s}}
		\int_{\mathbb{R}^{N-1}}^{}
		\GRD\parentheses{t-s , x-z} \;
		v|_{y=0}\parentheses{s,z} \;
		dz \,
	ds,
\end{alignat}
where
\begin{itemize}
	\item $V=V(t,X)$ is the solution to the  Cauchy problem
\begin{equation}\label{eq_pour_V}
\left\{
\begin{array}{ll}
	\partial_{t} V = d \Delta V,
&\quad t>0, \; x \in \R^{N-1}, \; y>0, \vspace{4pt}\\
	\nu V|_{y=0} - d \partial_{y} V|_{y=0} = 0,
&\quad t>0, \; x \in \R^{N-1}, \vspace{4pt}\\
	V|_{t=0} = v_{0},
&\quad \hspace{11.225mm} x \in \R^{N-1}, \; y>0,
\end{array}
\right.
\end{equation}
\item $U=U(t,x)$ is the solution to the  Cauchy problem
\begin{equation}\label{eq_pour_U}
\left\{
\begin{array}{ll}
	\partial_{t} U = D \Delta U,
&\quad t>0, \; x \in \R^{N-1}, \vspace{4pt}\\
	U|_{t=0} = u_{0},
&\quad \hspace{11.225mm} x \in \R^{N-1},
\end{array}
\right.
\end{equation}
	\item $\GRD=\GRD\parentheses{t,x}$ denotes the $\parentheses{N-1}$-dimensional Heat kernel with diffusion $D$, that is
\begin{equation}
\label{heat-kernel}
\GRD\parentheses{t,x}
\=
\frac{1}{\parentheses{4\pi Dt}^{\frac{N-1}{2}}}
e^{-\frac{\vertii{x}{}^{2}}{4Dt}},
\end{equation}
	\item and $\Lambda=\Lambda(t,x,y)$ is defined as
\begin{equation}
\label{def:Lambda}
\Lambda\parentheses{t,x,y} \=
\frac{e^{-\frac{y^{2}}{4dt}}}{\parentheses{2\pi}^{N-1}}
\int_{\mathbb{R}^{N-1}}^{}
	\Big[ a\alpha\Phi_{\alpha}+b\beta\Phi_{\beta}+c\gamma\Phi_{\gamma}\Big] \parentheses{t,\xi,y}
	\;
	e^{-dt\vertii{\xi}{}^{2} + i \xi \cdot x}
	\;
d\xi,
\end{equation}
with $\parentheses{\alpha,\beta,\gamma}=\parentheses{\alpha,\beta,\gamma}\parentheses{\xi}$ being the three complex roots of the $\delta$-indexed polynomials
\begin{equation}
\label{def:poly}
P_{\delta}\parentheses{\sigma} \=
\sigma^{3} + \frac{\nu}{\sqrt{d}} \, \sigma^{2} + \parentheses{\mu + \delta} \, \sigma + \frac{\nu \, \delta}{\sqrt{d}},
\qquad
\qquad
\text{with }\delta \= \parentheses{D-d}\vertii{\xi}{}^{2},
\end{equation}
$(a,b,c)=(a,b,c)(\xi)$ being given by
\begin{equation}\label{def:abc}
a\=\frac{1}{\parentheses{\alpha-\beta}\parentheses{\alpha-\gamma}},
\qquad\quad 
b\=\frac{1}{\parentheses{\beta-\alpha}\parentheses{\beta-\gamma}},
\qquad\quad 
c\=\frac{1}{\parentheses{\gamma-\alpha}\parentheses{\gamma-\beta}},
\end{equation}
and for $\bullet \in \ensemble{\alpha,\beta,\gamma}$,
\begin{equation}\label{def:phi}
\Phi_{\bullet}\parentheses{t,\xi,y} \=
\frac{\Erfc}{\Gamma}\Parentheses{\frac{-2\bullet\sqrt{d}\, t + y}{2\sqrt{dt}}},
\end{equation}
where $\Gamma\parentheses{\ell} \= e^{-\ell^{\,2}}$, and $\Erfc\!$ is the \textnormal{complementary error function}, whose definition is recalled in Appendix \ref{ss:appendix_ERFC}.
\end{itemize}
\end{theorem}

Let us give some remarks on Theorem \ref{th:heat-eq-FR-space-SOLUTION}. First, the theorem gives the unique {\em bounded} solution. Indeed, there could be {\em non-physical} solutions that change sign and grow very fast at infinity: this fact was proved for the Heat equation by Tychonoff \cite{Tyc}, see also \cite[Chapter 9]{Smo-book}. To get rid of such solutions, and so, to ask uniqueness, one needs to impose very loose  restriction on the growth at infinity; here we choose boundedness for the sake of simplicity.

Next, the Cauchy problem \eqref{eq_pour_U} is nothing else than the Heat equation in the hyperplane $\R^{N-1}$  and both the expression and the decay rate of its solution $U=U(t,x)$ are very well-known. On the other hand,  the Cauchy problem \eqref{eq_pour_V} is the Heat equation in the half-space $\R_+^N$ with Robin boundary conditions, and the understanding of its solution $V=V(t,X)$ is an important preliminary, that we recall in Section \ref{s:heat}.

Last, but not least, the \lq\lq migration-kernel'' $\Lambda=\Lambda(t,x,y)$ is the keystone to write the solution to the field-road diffusion model. Its main role is to describe the in-flux of individuals {\it in} the field {\it from} the road via the last two terms of $v$ in \eqref{v-sol-FR}, whose form  evokes a  Duhamel's formula. We may give a physical sense to $\Lambda$ by remarking that it is the \lq\lq $v$-component'' of the solution to \eqref{syst} starting from $\parentheses{v_{0},u_{0}}\equiv \parentheses{0,\frac{\sqrt{d}}{\mu}\delta_{x=0}}$. Notice also that, having in mind the acceleration phenomenon \cite{Ber-Roq-Ros-13-1, Ber-Roq-Ros-13-2} when $D>2d$, $\Lambda$ is the only term in the expression of $v$ that involves the constant $D$. However, the expression of $\Lambda(t,x,y)$ is quite intricate. As a consequence of our Fourier/Laplace transform approach, it involves an integral over $\xi \in \R^{N-1}$. Furthermore, we show in Appendix \ref{ss:appendix_P_delta} that, for almost all $\xi \in \R^{N-1}$, the three complex roots $\alpha=\alpha(\xi)$, $\beta=\beta(\xi)$ and $\gamma=\gamma(\xi)$ of the polynomials $P_\delta$ are distinct, which allows to define $a=a(\xi)$, $b=b(\xi)$ and $c=c(\xi)$ in \eqref{def:abc}. Further details will appear in Section \ref{s:fund} and \ref{s:decay} and in Appendix \ref{ss:appendix_P_delta}.

\medskip

Our second main contribution is to estimate the decay rate of the solution to the linear field-road system \eqref{syst} starting from the datum, say
\begin{equation}\label{initial_data-bis}
\left\{
\begin{array}{lll}
	v|_{t=0} = v_{0},
	&\quad X \in \mathbb{R}^{N}_{+},
	&\quad v_{0}\in L^{\infty}\parentheses{\mathbb{R}^{N}_{+}} \text{ is nonnegative, and compactly supported}, \vspace{4pt}\\
		u|_{t=0} = u_{0},
	&\quad x \in \R^{N-1},
	&\quad u_{0}\in  L^{\infty}\parentheses{\mathbb{R}^{N-1}} \text{ is nonnegative, and compactly supported}. 
\end{array}
\right.
\end{equation}
In the statement below, the notation $B\lesssim B'$ means there is a constant $k=k(N,d,D,\nu,\mu)>0$ such that $B\leq k B'$. 

\begin{theorem}[Decay rate of the solutions to the \textit{linear} field-road system]\label{th:decay}
Assume $D> d$ and let $\parentheses{v,u}$ be the bounded solution to the Cauchy problem \eqref{syst}\?--\,\eqref{initial_data-bis}. Then 
\begin{equation}\label{decay-v}
\vertii{v\parentheses{t,\point}}{L^{\infty}\parentheses{\mathbb{R}^{N}_{+}}}
\lesssim
\frac{\Cvzero \ln\parentheses{1+t} + \Cvuzero }{(1+t)^{\frac{N}{2}}},
\qquad \forall t > 0,
\end{equation}
\begin{equation}\label{decay-u}
\vertii{u\parentheses{t,\point}}{L^{\infty}\parentheses{\mathbb{R}^{N-1}}}
\lesssim
\frac{\Cvzero \ln\parentheses{1+t} + \Cvuzero }{(1+t)^{\frac{N}{2}}},
\qquad \forall  t> 0,
\end{equation}
for some nonnegative constant $\Cvzero$ depending only on $v_0$, and some nonnegative constant $\Cvuzero$ \break depending on both  $v_0$ and $u_0$.
\end{theorem}

The starting point to prove this upper estimate is the explicit expression of $v$ and $u$ given  in Theorem \ref{th:heat-eq-FR-space-SOLUTION}. The main difficulty is to control the \lq\lq migration-kernel'' $\Lambda$ defined in \eqref{def:Lambda}\,---\,\eqref{def:phi} which, as explained above, is the core of the field-road model. To do so, we shall rely on a rather technical result, namely Lemma \ref{lem:controle-L-infty-Phi}, whose proof needs a full understanding of the $\delta$-indexed polynomials $P_{\delta}$ in \eqref{def:poly}, achieved in Appendix \ref{ss:appendix_P_delta}. Let us emphasize that the expression of $\Lambda$ we own, namely \eqref{def:Lambda}\,---\,\eqref{def:phi}, is highly conditioned by the fact that, at some point, we  use a decomposition into partial fractions to match up with known forms of Laplace-transforms. This decomposition makes appear the \MAT{artefacts} $a$, $b$ and $c$ which blow up as the roots of $P_{\delta}$ collide for singular values of $\delta$,  see \eqref{def:abc}. Fortunately, due to the compensating structure of $a\alpha\Phi_{\alpha}+b\beta\Phi_{\beta}+c\gamma\Phi_{\gamma}$ and as confirmed by the technical Lemma \ref{lem:controle-L-infty-Phi}, these singularities are \lq\lq artificial''. 

By following the proof in Section \ref{s:decay}, one can track the values of $\Cvzero$ and $\Cvuzero$ and check, in particular, that $v_{0}\equiv 0$ implies $\Cvzero=0$. As a result, if there is initially no individual in the field, the estimate of the decay rate becomes of the magnitude $\mathcal{O}((1+t)^{-N/2})$, which corresponds to the decay rate of the Heat equation in the half space $\R^N_+$ with Neumann boundary conditions --- see Theorem \ref{th:heat-eq-half-space}. This is consistent with the fact that the boundary conditions of  the field-road model are of the exchange type so that, roughly speaking, individuals are \lq\lq not lost'' but \lq\lq stored'' in the road. On the other hand, if $v_0\not \equiv 0$, our estimate of the decay rate becomes of the magnitude $\mathcal{O}((1+t)^{-N/2}\ln(1+t))$, which is very slightly \lq\lq less good''. We believe that the actual decay rate is always of the magnitude $\mathcal{O}((1+t)^{-N/2})$: to deal with the terms $P(t,x,y)$ and $Q(t,x,y)$ appearing in \eqref{truc}, we rely on the uniform estimate \eqref{controle-L-infty-Phi} of  the rather intricate function $\Phi\parentheses{t,\xi,y}$ appearing in Lemma \ref{lem:controle-L-infty-Phi}. The use of this uniform estimate is the cause of the logarithmic term appearing later, but refined estimates seem delicate to reach.  Notice however that, in order to study the Fujita blow-up phenomenon, the logarithmic term should only play a (bad) role in the so-called {\it critical case}.

\medskip

The organization of the present paper is as follows. In Section \ref{s:heat}, we recall some facts about the fundamental solution of the Heat equation in a half-space, that will prove useful in the sequel. In particular, we give a complete derivation of the expression of the fundamental solution in the case of Robin boundary conditions, thus introducing the strategy of proof we will employ to study the complete system \eqref{syst} in Section \ref{s:fund}, where we compute the explicit form of the solutions of the linear field-road system, while their decay rate is estimated in Section \ref{s:decay}. Last, in Section \ref{s:num}, we present some numerical simulations which not only \lq\lq validate'' Theorem \ref{th:decay} but also explore some open problems related to the level sets of the solution, or to the sign of the {\it flux} of individuals from the road to the field.

%%%%%%%%%%%%%%%%%%%%%%%%%%%%%%%%%%%%%%%%%%%%%%%%%%%%%%%%%%%%%%%%%%%%%%%%%%
%%%%%%%%%%%%%%%%%%%%%%%%%%%%%%%%%%%%%%%%%%%%%%%%%%%%%%%%%%%%%%%%%%%%%%%%%%
%%%%%%%%%%%%%%%%%%%%%%%%%%%%%%%%%%%%%%%%%%%%%%%%%%%%%%%%%%%%%%%%%%%%%%%%%%

\section{The Heat equation in the half-space}\label{s:heat}

As an important preliminary, we discuss in this section the Heat equation  in the half-space, in particular the explicit expression and the decay rate of the $L^\infty$ norm of the solution to the associated Cauchy problem with different boundary conditions. These facts are very classical for the Neumann and Dirichlet boundary conditions, but the details for the case of Robin boundary conditions are not so easy to find. We therefore take this opportunity to present the main arguments, by using a Fourier/Laplace transform strategy, which we use again in the next section to deal with the whole field-road system \eqref{syst}.

\medskip

Let $N \geq 1$ be an integer. We recall that a generic point $X\in \R^{N}$ will be written ${X=(x,y)}$ with $x\in \R^{N-1}$ and $y\in \R$. For the \lq\lq variable of integration'', we reserve the notation ${Z=(z,\omega)\in \R^N}$, with $z\in \R^{N-1}$ and $\omega \in \R$. For $d>0$, we denote the $N$-dimensional Heat kernel with diffusion $d$,
\begin{equation}
\label{heat-kernel-bis}
G(t,X)
\=
\frac{1}{(4\pi dt)^{\frac{N}{2}}}
e^{-\frac{\vertii{X}{}^{2}}{4dt}}
=\frac{1}{(4\pi dt)^{\frac{N}{2}}}e^{-\frac{\vertii{x}{}^{2}+y^{2}}{4dt}},
\end{equation}
and finally, we denote the upper half-space $\R^{N}_{+} \= \R^{N-1} \times (0,+\infty)$, its boundary being the hyperplane ${\partial\R^N_{+} = \R^{N-1}\times\{0\}}$. 

For  $0\leq \alfa \leq 1$, we thus consider the Cauchy problem 
\begin{equation}\label{eq-heat}
\left\{
\begin{array}{ll}
	\partial_{t} v = d \Delta v,
&\quad t>0, \; X \in \R_{+}^{N}, \vspace{4pt}\\
	\alfa v + \parentheses{1-\alfa} d \partial_{n} v = 0,
&\quad t>0, \; X \in \partial \R^{N}_{+},\vspace{4pt}\\
	v|_{t=0} = v_{0},
&\quad \hspace{11.225mm} X \in \R_{+}^{N},
\end{array}
\right.
\end{equation}
where, say,  $v_{0}\in  L^\infty(\R_+^N)$. Here $n \= (0_{\R^{N-1}},-1)$ and the boundary condition may be recast
$$
\hspace{3.165mm}
\alfa v - \parentheses{1-\alfa} d \partial_{y} v = 0,
\quad \hspace{3.165mm} t>0, \; X \in \partial \R^{N}_{+}.
$$

In this section, the notation $B \lesssim B'$ means there is a constant $k=k(N,d,\alfa)>0$ such that $B\leq k B'$.

\begin{theorem}[Heat equation in the half-space]\label{th:heat-eq-half-space} The solution to the Cauchy problem \eqref{eq-heat} is
\begin{equation}
\label{v-sol}
v(t,X) =
\int_{\R^{N}_+}
	H_\alfa (t,X,Z) \, v_{0}(Z) \,
dZ,
\end{equation}
for some $H_{\alfa}(t,X,Z)$ precised below.
\begin{itemize}
\refstepcounter{HEHalfSpace}\label{HE-Half-Space-Neumann}
\item [$(\theHEHalfSpace)$] When $\alfa=0$ (Neumann case), we have
$$
H_{0}(t,X,Z) =
G(t,x-z,y-\omega)
+
G(t,x-z,y+\omega),
$$
and
$$
\vertii{v(t,\point)}{{L^{\infty}(\mathbb{R}^{N}_{+})}}
\lesssim
	\frac  {\int_{\R_+^N }\vert v_{0}(Z) \vert \, dZ}  {t^{\frac{N}{2}}},
\qquad \forall t>0.
$$
\refstepcounter{HEHalfSpace}\label{HE-Half-Space-Dirichlet}
\item [$(\theHEHalfSpace)$] When $\alfa=1$ (Dirichlet case), we have
$$
H_{1}(t,X,Z) = G(t,x-z,y-\omega) - G(t,x-z,y+\omega),
$$
and
$$
\vertii{v(t,\point)}{{L^{\infty}(\mathbb{R}^{N}_{+})}}
\lesssim
	\frac  {\int_{\R_+^N }\omega \, \vert v_{0}(Z)\vert  \, dZ}  {t^{\frac{N+1}{2}}},
\qquad \forall t>0.
$$
\refstepcounter{HEHalfSpace}\label{HE-Half-Space-Robin}
\item [$(\theHEHalfSpace)$] When $0<\alfa<1$ (Robin case), we have 
\begin{align}
\hspace{-1cm}
H_\alfa(t,X,Z)
	&= H_{0}(t,X,Z)
	-2\sqrt{\pi}A\sqrt{dt} \,
	G(t,x-z,y+\omega)\,
	\frac{\Erfc}{\Gamma}
	\Parentheses{\frac{2Adt+y+\omega}{2\sqrt{dt}}},
	\label{H-robin-zero}\\
	&\hspace*{-1.1804cm} =
	H_{1}(t,X,Z)
	+2G(t,x-z,y+\omega)
	\Parentheses{
	1-\sqrt{\pi}A\sqrt{dt}\,
	\frac{\Erfc}{\Gamma}
	\Parentheses{\frac{2Adt+y+\omega}{2\sqrt{dt}}}
	            },\label{H-robin-un}
\end{align}
where
$$
A \= \frac{\alfa}{d\parentheses{1-\alfa}},
$$
$\Gamma(\ell) \= e^{-\ell^{\,2}}$, and the definition of $\Erfc\!$ is recalled in Appendix \ref{ss:appendix_ERFC}. Furthermore,
\begin{equation}\label{decay-robin}
\vertii{v(t,\point)}{{L^{\infty}(\mathbb{R}^{N}_{+})}}
\lesssim
	\frac  {\int_{\R_+^N}(1+\omega)\vert v_{0}(Z)\vert\, dZ}  {t^{\frac{N+1}{2}}},
\qquad \forall t>0,
\end{equation}
and
\begin{equation}\label{decay-robin-bis}
\vertii{v(t,\point)}{{L^{\infty}(\mathbb{R}^{N}_{+})}}
\lesssim
	\frac{\CvzeroSHARP}{(1+t)^{\frac{N+1}{2}}},
\qquad \forall t> 0,
\end{equation}
where $\CvzeroSHARP \= \left(\vertii{v_0}{{L^{\infty}(\mathbb{R}^{N}_{+})}}^{\frac{2}{N+1}}+\left(\int_{\R_+^N}(1+\omega)\vert v_{0}(Z)\vert\, dZ\right)^{\frac{2}{N+1}}\right)^{\frac{N+1}{2}}$\hspace{-0.1cm}.
\end{itemize}
\end{theorem}

This theorem is rather classical, in particular points $(i)$ and $(ii)$. A possible way to prove it is to use some continuation arguments. For instance, to prove $(\ref{HE-Half-Space-Neumann})$, one defines $\widetilde{v}(t,X)$ as the solution of the Heat equation in the whole of $\R^N$ arising from the initial datum $$
\widetilde{v}_{0}(x,y) \=
\begin{cases}
	v_{0}(x,y)  	& \text{ if } y>0, \\
	v_{0}(x,-y) 	& \text{ if } y<0.
\end{cases}
$$
Then, by uniqueness of the (bounded) solutions of the Heat equation, $\widetilde{v}$ remains even with respect to $y$,  is $\mathcal{C}^{\infty}$ in $\R^N\times(0,+\infty)$ by usual regularity results on parabolic equations. Thus it satisfies $\partial_{y} \widetilde{v}(t,x,0)=0$ for all $t>0$ and all $x\in \R^{N-1}$. Hence, $v(t,x,y)$ stands as the restriction of $\widetilde{v}(t,x,y)$ to the upper-half-space, and we have
\begin{align*}
 	v(t,X)
	&= \int_{\R^N} G(t,X-Z) \, \widetilde{v}_{0}(Z) \, dZ\\
 	&= \int_{\R^N_+}
		\Big[G(t,x-z,y-\omega) + G(t,x-z,y+\omega)\Big] \,
		v_{0}(z,\omega) \,
	   dZ,
\end{align*}
from which we deduce the expression of $H_{0}$ and the decay rate. The point $(\ref{HE-Half-Space-Dirichlet})$ can be proven in the same way by considering this time, for $\widetilde{v}_{0}$, the odd-continuation (with respect to $y$) of $v_{0}$.
A more tricky (and less known) continuation argument  can be used to reach the point $(\ref{HE-Half-Space-Robin})$. However, in view of the sequel, let us present now the detailed proof of $(\ref{HE-Half-Space-Robin})$ using rather a Fourier/Laplace transform strategy.

In the following, the $x$-Fourier transform of $u$ shall be denoted by $\hat{u}$, the $t$-Laplace one by $\frown{u}$ and the $x$-Fourier/$t$-Laplace one by $\frownhat{u}$. We refer to Appendix \ref{ss:appendix_Transforms} for conventions and notations related to these transforms.

\begin{proof} [Proof of Theorem \ref{th:heat-eq-half-space}, $(\ref{HE-Half-Space-Robin})$] 
 Applying the $t$-Laplace transform to the first equation in \eqref{eq-heat}, namely $\partial_t v(t,x,y)=d(\Delta_x+\partial_{yy}) v(t,x,y)$,  yields the elliptic PDE
$$
d\Parentheses{\Delta_{x}\frown{v}(s,x,y) + \partial_{yy} \frown{v}(s,x,y)}- s\frown{v}(s,x,y)
= - v_{0}(x,y),
\qquad (x,y)\in\mathbb{R}^{N}_{+},
$$
where $s>0$ acts as a parameter.  Then by applying the $x$-Fourier transform, we reach the linear second order ODE
\begin{equation}\label{3_2_ODE_sur_frownhat_v}
d \partial_{yy} \frownhat{v}\parentheses{s,\xi,y}-
	\Parentheses{s+d\vertii{\xi}{}^{2}}\frownhat{v}\parentheses{s,\xi,y}
	=
	-\hat{\,v_{0}}\parentheses{\xi ,y},
\qquad y>0,
\end{equation}
where $s>0$ and $\xi\in \R^{N-1}$ act as parameters. Setting $\sigma \= \sqrt{s+d\vertii{\xi}{}^{2}}$, solutions of \eqref{3_2_ODE_sur_frownhat_v} may be written
\begin{multline*}
\frownhat{v}\parentheses{s,\xi,y} =
	e^{\frac{\sigma}{\sqrt{d}} \, y}
	\Parentheses{
	C_1-\frac{1}{2\sqrt{d}\sigma}
	\int_{0}^{y}
		e^{-\frac{\sigma}{\sqrt{d}} \, \omega} \, \hat{\,v_{0}}\parentheses{\xi,\omega}
	d\omega
	            }
\\	
	+
	e^{-\frac{\sigma}{\sqrt{d}} \, y}
	\Parentheses{
	C_2+\frac{1}{2\sqrt{d}\sigma}
	\int_{0}^{y}
		e^{\frac{\sigma}{\sqrt{d}} \, \omega} \, \hat{\,v_{0}}\parentheses{\xi,\omega}
	d\omega
	            },
\end{multline*}
where $C_1$ and $C_2$ are constants (with respect to $y$) to be determined. In order to insure ${\lim\limits_{y\rightarrow +\infty}\frownhat{v}\parentheses{s,\xi ,y}=0}$, one needs
\begin{equation}\label{def:B}
C_{1} =
\frac{1}{2\sqrt{d}\sigma}
	\int_{0}^{+\infty}
		e^{-\frac{\sigma}{\sqrt{d}}\omega} \, \hat{\,v_{0}}\parentheses{\xi,\omega}	d\omega.
\end{equation}
Next, since
$$
\frownhat{v}\parentheses{s,\xi,0} = C_{1}+C_{2}
\qquad\quad\text{and}\quad\qquad 
\partial_{y} \frownhat{v}\parentheses{s,\xi,0} = \frac{\sigma}{\sqrt{d}}\parentheses{C_1-C_2},
$$
the Robin boundary conditions enforce
\begin{equation}
\label{def:C}
C_2
=
\Parentheses{1 - \frac{2A\sqrt{d}}{\sigma + A\sqrt{d}}} C_1.
\end{equation}
This brings us to the following \textit{explicit} expression for $\frownhat{v}$:
$$
\frownhat{v}\parentheses{s,\xi,y} =
\frac{1}{2\sqrt{d}}
\int_{0}^{+\infty}
	\Parentheses{
		\frac{e^{-\frac{\sigma}{\sqrt{d}}\verti{y-\omega}}}{\sigma}
		+
		\frac{e^{-\frac{\sigma}{\sqrt{d}}\parentheses{y+\omega}}}{\sigma}
		-
		2A\sqrt{d}
		\frac{e^{-\frac{\sigma}{\sqrt{d}}\parentheses{y+\omega}}}{\sigma\parentheses{\sigma+A\sqrt{d}}}
		}
	\hat{v_{0}}\parentheses{\xi,\omega}
d\omega,
$$
where we recognize known forms of $t$-Laplace transforms evaluated at $\sigma^2=s+d\vertii{\xi}{}^{2}$ --- see Lemma \ref{lem:inverse-Laplace-Fourier}, $(\ref{A1-Laplace-E_sur_sig})$ and $(\ref{A1-Laplace-E_sur_sig_fois_sig_plus_b})$ --- so that
\begin{multline*}
\frownhat{v}\parentheses{s,\xi,y} =
\frac{1}{2\sqrt{d}}
\int_{0}^{+\infty}
	{\Lb}\Bigg[
		t\mapsto
		\frac{e^{-\frac{\parentheses{y-\omega}^{2}}{4dt}}}{\sqrt{\pi t}}
		+
		\frac{e^{-\frac{\parentheses{y+\omega}^{2}}{4dt}}}{\sqrt{\pi t}}
\\		
		-2A\sqrt{d}
		\frac{\Erfc}{\Gamma}\Parentheses{\frac{2Adt+y+\omega}{2\sqrt{dt}}}
		e^{-\frac{\parentheses{y+\omega}^{2}}{4dt}}
	      \Bigg]
	      \parentheses{s+d\vertii{\xi}{}^{2}}
	\;\,\hat{v_{0}}\parentheses{\xi,\omega}
d\omega.
\end{multline*}
Using Lemma \ref{lem:propri-Laplace-Fourier} $(\ref{A2-Delay_th})$, we can now get rid of the $t$-Laplace transform and get
\begin{align*}
\hat{v}\parentheses{t,\xi,y} &= \frac{1}{2\sqrt{d}}
\int_{0}^{+\infty}
\Bigg[
	\frac{e^{-\frac{\parentheses{y-\omega}^{2}}{4dt}}}{\sqrt{\pi t}}
	+
	\frac{e^{-\frac{\parentheses{y+\omega}^{2}}{4dt}}}{\sqrt{\pi t}}
		\Parentheses{1-2\sqrt{\pi}A\sqrt{dt}
		\frac{\Erfc}{\Gamma}\Parentheses{\frac{2Adt+y+\omega}{2\sqrt{dt}}}
		}
\Bigg]
\\
	&\hspace{104.195mm}e^{-dt\vertii{\xi}{}^{2}}
 \, \hat{v_{0}}\parentheses{\xi,\omega}
d\omega.
\end{align*}
Finally, note that
$$
e^{-dt\vertii{\xi}{}^{2}} \, \hat{v_{0}}\parentheses{\xi,\omega}
=
{\Fb}\crochets{\frac{e^{-\frac{\vertii{\point}{}^{2}}{4dt}}}{\Parentheses{4\pi dt}^{\frac{N-1}{2}}}} \!\! \parentheses{\xi} \times {\Fb}\crochets{v_{0}\parentheses{\point,\omega}}\!\parentheses{\xi}
=
{\Fb}\crochets{\frac{e^{-\frac{\vertii{\point}{}^{2}}{4dt}}}{\Parentheses{4\pi dt}^{\frac{N-1}{2}}} \ast v_{0}\parentheses{\point,\omega}}\!\!\parentheses{\xi},
$$
whence we can also get rid of the $x$-Fourier transform and reach
\begin{align}\label{3_2_explicit_v}
\hspace{-2mm}v\parentheses{t,x,y} \! &= \!\!
\int_{\mathbb{R}^{N}_{+}}^{}\! \!
	\crochets{
		\frac{e^{-\frac{\vertii{x-z}{}^{2}}{4dt}}}{\Parentheses{4\pi dt}^{\frac{N-1}{2}}}
		\Parentheses{
		\frac{e^{-\frac{\parentheses{y-\omega}^{2}}{4dt}}}{\sqrt{4\pi dt}}
	+
		\frac{e^{-\frac{\parentheses{y+\omega}^{2}}{4dt}}}{\sqrt{4\pi dt}}
		\Parentheses{1-2\sqrt{\pi}A\sqrt{dt}
		\frac{\Erfc}{\Gamma}\Parentheses{\frac{2Adt+y+\omega}{2\sqrt{dt}}}
		}}
		}
\nonumber\\
	&\hspace{110.573mm} v_{0}\parentheses{Z} \;
dZ.
\end{align}
Expressions \eqref{H-robin-zero} and \eqref{H-robin-un} of $H_{\alfa}$ are immediately deduced from \eqref{3_2_explicit_v}.

%Contrôle du H_{\alfa} :
We now turn to the decay rate. It is readily seen that
$$
-\sqrt{\pi} \, \frac{\Erfc}{\Gamma}(\ell) \leq - \frac{1}{1+\ell },
\qquad \forall \ell \geq 0.
$$
As a consequence, for all $t>0$, all $y \geq 0$, and all $\omega \geq 0$,$$
-\sqrt{\pi} \, \frac{\Erfc}{\Gamma}\Parentheses{\frac{2Adt+y+\omega}{2\sqrt{dt}}}
\leq
-\frac{2\sqrt{dt}}{2Adt+2\sqrt{dt}+y+\omega}.
$$
Using this into \eqref{H-robin-un}, we see that, for all $t>0$, all $X\in\R^{N}_{+}$, and all $Z\in\R^{N}_{+}$,
\begin{eqnarray}
0 < H_{\alfa}(t,X,Z) - H_{1}(t,X,Z)
	&\leq&
		2 G(t,x-z,y+\omega)
		\,
		\frac {2\sqrt{dt}+y+\omega} {2Adt+2\sqrt{dt}+y+\omega}
		\nonumber\\
	&\lesssim&
		\frac{m(t)}{t^{\frac{N}{2}}},\label{presque}
\end{eqnarray}
where 
$$
m(t) \= \sup_{r\geq 0} 
	\Parentheses{
		\frac{2\sqrt{dt}+r}{2Adt+2\sqrt{dt}+r}
		\,
		e^{-\frac{r^{2}}{4dt}}
	}.
$$
An elementary analysis shows that
$$
m(t)
= \max_{0\leq r \leq \sqrt{dt}} 
	\Parentheses{
		\frac{2\sqrt{dt}+r}{2Adt+2\sqrt{dt}+r}
		\,
		e^{-\frac{r^{2}}{4dt}}
	}
\leq
\max_{0\leq r \leq \sqrt{dt}} 
	\Parentheses{
		\frac{2\sqrt{dt}+r}{2Adt+2\sqrt{dt}+r}
	}
%=
%\frac{3}{2A\sqrt{dt}+3}
\lesssim
	\frac{1}{\sqrt{t}}.
$$
From this, \eqref{presque}, and the case $(\ref{HE-Half-Space-Dirichlet})$, we obtain the decay rate as stated in \eqref{decay-robin}. Combining \eqref{decay-robin} and the fact that $\vertii{v(t,\point)}{{L^{\infty}(\mathbb{R}^{N}_{+})}}\leq \vertii{v_0}{{L^{\infty}(\mathbb{R}^{N}_{+})}}$, one can check that \eqref{decay-robin-bis} holds true.
\end{proof}

Let us conclude this section with the following direct consequence of Theorem \ref{th:heat-eq-half-space}.

\begin{corollary}\label{cor:split}
For any $0\leq \theta \leq 1$, we have 
$$
H_{\alfa}\parentheses{t,X,Z} =
\frac{e^{-\frac{\vertii{x-z}{}^{2}}{4dt}}}{\Parentheses{4\pi dt}^{\frac{N-1}{2}}}
\times
\HalfaONE\parentheses{t,y,\omega},
$$
where $\HalfaONE$ denotes $H_{\alfa}$ for $N=1$.
\end{corollary}

%%%%%%%%%%%%%%%%%%%%%%%%%%%%%%%%%%%%%%%%%%%%%%%%%%%%%%%%%%%%%%%%%%%%%%%%%%
%%%%%%%%%%%%%%%%%%%%%%%%%%%%%%%%%%%%%%%%%%%%%%%%%%%%%%%%%%%%%%%%%%%%%%%%%%
%%%%%%%%%%%%%%%%%%%%%%%%%%%%%%%%%%%%%%%%%%%%%%%%%%%%%%%%%%%%%%%%%%%%%%%%%%

\section{The solution to the field-road model Cauchy problem}\label{s:fund}

In this section, we explicitly obtain the bounded solution to the Cauchy problem \eqref{syst}\,--\,\eqref{initial_data}. Before starting, we recall that the $x$-Fourier transform of $u$ is denoted by $\hat{u}$, the $t$-Laplace one by $\frown{u}$ and the $x$-Fourier/$t$-Laplace one by $\frownhat{u}$. We refer to Appendix \ref{ss:appendix_Transforms} for conventions and notations related to these transforms.

\begin{proof} [Proof of Theorem \ref{th:heat-eq-FR-space-SOLUTION}] 
We start with establishing expression \eqref{u-sol-FR} for $u=u(t,x)$ which is the simplest to get. To do so, we apply the $t$-Laplace and then the $x$-Fourier transforms to  the third equation in \eqref{syst}, namely $\partial_{t}u\parentheses{t,x} = D\Delta_{x} u\parentheses{t,x} + \nu v|_{y=0}\parentheses{t,x} - \mu u\parentheses{t,x}$. This yields the algebraic relation
$$
-D \vertii{\xi}{}^{2} \; \frownhat{u}\parentheses{s,\xi}
+ \nu \; \frownhat{v}\?\?|_{y=0}\parentheses{s,\xi}
-\mu \; \frownhat{u}\parentheses{s,\xi}
	=
- \hat{\,u_{0}}\parentheses{\xi} + s \; \frownhat{u}\parentheses{s,\xi},
\qquad s>0, \; \xi\in\mathbb{R}^{N-1},
$$
which is recast
\begin{equation}
\label{u-frownhat}
\frownhat{u}\parentheses{s,\xi} =
\frac{\hat{\,u_{0}}\parentheses{\xi} + \nu \; \frownhat{v}\?\?|_{y=0}\parentheses{s,\xi}}{\mu+s+D\vertii{\xi}{}^{2}},
\qquad s>0, \; \xi\in\mathbb{R}^{N-1}.
\end{equation}
It is now possible to reach \eqref{u-sol-FR} from \eqref{u-frownhat} by using properties and known forms of $t$-Laplace and $x$-Fourier transforms --- see Lemmas \ref{lem:inverse-Laplace-Fourier}, $(\ref{A1-Fourier-Gaussienne})$ and $(\ref{A1-Laplace-un-sur-a-plus-s})$, and \ref{lem:propri-Laplace-Fourier}, $(\ref{A2-Fourier-convolee})$ and $(\ref{A2-Laplace-convolee})$:
\begin{align*}
 	\frownhat{u}\parentheses{s,\xi}
 	&={\Lb}\crochets{t\mapsto e^{-\parentheses{\mu + D \vertii{\xi}{}^{2}} t} \, \hat{\,u_{0}}\parentheses{\xi}}\parentheses{s}
 	+ \nu {\Lb}\crochets{t\mapsto e^{-\parentheses{\mu+D\vertii{\xi}{}^{2}}t}}\parentheses{s} \times {\Lb}\crochets{t\mapsto\hat{v}\?|_{y=0}\parentheses{t,\xi}}\parentheses{s}\\
 	&={\Lb}\crochets{t\mapsto e^{-\parentheses{\mu + D \vertii{\xi}{}^{2}} t} \, \hat{\,u_{0}}\parentheses{\xi}}\parentheses{s}
 	+ \nu {\Lb}\crochets{t\mapsto \int_{0}^{t} e^{-\parentheses{\mu+D\vertii{\xi}{}^{2}}\parentheses{t-\tau}} \; \hat{v}\?|_{y=0}\parentheses{\tau,\xi} \, d\tau}\parentheses{s}\\
 	&={\Lb}{\Fb}\crochets{\parentheses{t,x} \mapsto e^{-\mu t}\crochets{u_{0}\ast \GRD\parentheses{t,\point}}\parentheses{x}}\parentheses{s,\xi}\\ 	&\hspace{29.710mm}+\nonumber \nu {\Lb}{\Fb}\crochets{\parentheses{t,x}\mapsto \int_{0}^{t} e^{-\mu\parentheses{t-\tau}}\crochets{v|_{y=0}\parentheses{\tau,\point}\ast\GRD\parentheses{t-\tau,\point}}\parentheses{x} \, d\tau}\parentheses{s,\xi}\\
 	&={\Lb}{\Fb}\crochets{\parentheses{t,x} \mapsto e^{-\mu t}\Parentheses{U\parentheses{t,x} + \nu \int_{0}^{t} e^{\mu\tau} \int_{\mathbb{R}^{N-1}}^{} \GRD\parentheses{t-\tau,x-z}v|_{y=0}\parentheses{\tau,z} \, dz \, d\tau}}\parentheses{s,\xi}
\end{align*}
which provides \eqref{u-sol-FR}.

We now turn to the expression \eqref{v-sol-FR} for $v=v(t,X)$. For convenience let us set
\begin{equation}
\label{def:sigma}
\sigma \= \sqrt{s+d\vertii{\xi}{}^{2}}
\qquad\quad 
\text{and}
\qquad\quad 
\Sigma \= \sqrt{s+D\vertii{\xi}{}^{2}}.
\end{equation}
As in Section \ref{s:heat}, applying $t$-Laplace then $x$-Fourier transforms on the first equation in \eqref{syst}, namely $\partial_{t}v\parentheses{t,x,y}=d\parentheses{\Delta_{x}+\partial_{yy}}v\parentheses{t,x,y}$, leads to the linear second order ODE (where $s>0$ and $\xi\in \R^{N-1}$ serve as parameters)
\begin{equation}\label{4_ODE_sur_frownhat_v}
d \partial_{yy} \frownhat{v}\parentheses{s,\xi,y}-
	\Parentheses{s+d\vertii{\xi}{}^{2}}\frownhat{v}\parentheses{s,\xi,y}
	=
	-\hat{\,v_{0}}\parentheses{\xi ,y},
\qquad y>0,
\end{equation}
whose solutions are
\begin{multline*}
\frownhat{v}\parentheses{s,\xi,y} =
	e^{\frac{\sigma}{\sqrt{d}} \, y}
	\Parentheses{
	C_{1}-\frac{1}{2\sqrt{d}\sigma}
	\int_{0}^{y}
		e^{-\frac{\sigma}{\sqrt{d}} \, \omega}\hat{\,v_{0}}\parentheses{\xi,\omega}
	d\omega
	            }
\\	
	+
	e^{-\frac{\sigma}{\sqrt{d}} \, y}
	\Parentheses{
	C_{2}+\frac{1}{2\sqrt{d}\sigma}
	\int_{0}^{y}
		e^{\frac{\sigma}{\sqrt{d}} \, \omega}\hat{\,v_{0}}\parentheses{\xi,\omega}
	d\omega
	            },
\end{multline*}
where $C_{1}$ and $C_{2}$ are constants (with respect to $y$) to be determined. In order to insure ${\lim\limits_{y\rightarrow +\infty}\frownhat{v}\parentheses{s,\xi ,y}=0}$, one needs
\begin{equation}
\label{def:C1}
C_{1} =
\frac{1}{2\sqrt{d}\sigma}
	\int_{0}^{+\infty}
		e^{-\frac{\sigma}{\sqrt{d}}\omega}\hat{\,v_{0}}\parentheses{\xi,\omega}	d\omega.
\end{equation}
Then, calling $A \= \frac{\nu}{d}$ and $B \= \frac{\mu}{d}$, the exchange condition --- second equation in \eqref{syst} --- is recast $Av|_{y=0}\parentheses{t,x} - \partial_{y}v|_{y=0}\parentheses{t,x}=Bu\parentheses{t,x}$; and since
$$
\frownhat{v}\parentheses{s,\xi,0} = C_{1}+C_{2}
\qquad\quad
\text{and}
\qquad\quad
\partial_{y} \frownhat{v}\parentheses{s,\xi,0} = \frac{\sigma}{\sqrt{d}}\parentheses{C_{1}-C_{2}},
$$
the exchange condition enforces
\begin{align}
 	C_{2} &= \Parentheses{1 - \frac{2A\sqrt{d}}{\sigma + A\sqrt{d}}} C_{1} + \frac{B\sqrt{d}}{\sigma+A\sqrt{d}} \; \frownhat{u} (s,\xi) \nonumber \\
 	&= \Parentheses{1 - \frac{2A\sqrt{d}}{\sigma + A\sqrt{d}}} C_{1} + B\sqrt{d} \, \frac{\hat{\,u_{0}}\parentheses{\xi} + \nu \; \frownhat{v}\?\?|_{y=0}\parentheses{s,\xi}}{\parentheses{\Sigma^{2}+\mu}\parentheses{\sigma+A\sqrt{d}}},\label{def:C2}
\end{align}
where we used the expression \eqref{u-frownhat} for $\frownhat{u}(s,\xi)$ and the definition of $\Sigma$ in \eqref{def:sigma}. Comparing \eqref{def:B}, \eqref{def:C} with \eqref{def:C1}, \eqref{def:C2}, we see that the \lq\lq deviation'' of $\,\frownhat{v}(s,\xi,y)$ (of the field-road model) from the solution to the Heat equation in the half-space (see Section \ref{s:heat}) stands in the second term in the right-hand side of \eqref{def:C2}. Hence, using the computations that have been performed in Section \ref{s:heat} and calling $V=V(t,X)$ the solution to \eqref{eq_pour_V}, we reach an \textit{implicit} expression for $\frownhat{v}(s,\xi,y)$, namely
\begin{equation}\label{v-frown-hat-implicit}
\frownhat{v}\parentheses{s,\xi,y} =
\frownhat{\? V}\parentheses{s,\xi,y} +
B\sqrt{d} \,
e^{-\frac{\sigma}{\sqrt{d}}y} \,
\frac{\hat{\,u_{0}}\parentheses{\xi} + \nu \; \frownhat{v}\?\?|_{y=0}\parentheses{s,\xi}}{\parentheses{\Sigma^{2}+\mu}\parentheses{\sigma+A\sqrt{d}}}.
\end{equation}
Evaluating now \eqref{v-frown-hat-implicit} at $y=0$ yields
$$
\frownhat{v}\?\?|_{y=0}\parentheses{s,\xi} =
\frownhat{\? V}|_{y=0}\parentheses{s,\xi} +
B\sqrt{d} \,
\frac{\hat{\,u_{0}}\parentheses{\xi} + \nu \; \frownhat{v}\?\?|_{y=0}\parentheses{s,\xi}}{\parentheses{\Sigma^{2}+\mu}\parentheses{\sigma+A\sqrt{d}}},
$$
whence
\begin{multline}\label{v-frown-y-equal-0}
\frownhat{v}\?\?|_{y=0}\parentheses{s,\xi} =
\frac{B\sqrt{d}}{\parentheses{\Sigma^{2}+\mu}\parentheses{\sigma+A\sqrt{d}}-\nu B\sqrt{d}}
\; \hat{\,u_{0}}\parentheses{\xi}\\
+
\frac{\parentheses{\Sigma^{2}+\mu}\parentheses{\sigma+A\sqrt{d}}}{\parentheses{\Sigma^{2}+\mu}\parentheses{\sigma+A\sqrt{d}}-\nu B\sqrt{d}}
\; \frownhat{\? V}|_{y=0}\parentheses{s,\xi}.
\end{multline}
Plugging \eqref{v-frown-y-equal-0} into \eqref{v-frown-hat-implicit} now yields the \textit{explicit} expression 
\begin{equation}\label{v-frown-hat-explicit}
\frownhat{v}\parentheses{s,\xi,y} =
\frownhat{\? V}\parentheses{s,\xi,y}
+ B\sqrt{d}
\frac{e^{-\frac{\sigma}{\sqrt{d}}y}}{\parentheses{\Sigma^{2}+\mu}\parentheses{\sigma+A\sqrt{d}}-\nu B\sqrt{d}}
\crochets{\hat{\,u_{0}}\parentheses{\xi} + \nu \frownhat{\? V}|_{y=0}\parentheses{s,\xi}}.
\end{equation}
Our task is now to take the inverse Fourier/Laplace transform of  the second term in the above right-hand-side. Letting
$$
\delta \= \parentheses{D-d}\vertii{\xi}{}^{2},
$$
one has $\Sigma^{2} = \sigma^{2}+\delta$. Then define
\begin{equation}
\label{return}
\frownhat{\Lambda}\parentheses{s,\xi,y} \=
\frac{e^{-\frac{\sigma}{\sqrt{d}}y}}{\parentheses{\Sigma^{2}+\mu}\parentheses{\sigma+A\sqrt{d}}-\nu B\sqrt{d}}
=
\frac{e^{-\frac{\sigma}{\sqrt{d}}y}}{\sigma^{3} + A\sqrt{d}\sigma^{2} + \parentheses{\mu+\delta}\sigma +A\sqrt{d}\delta}.
\end{equation}
To match up $\, \frownhat{\Lambda}(s,\xi,y)$ with known $t$-Laplace transforms, we need to expand \eqref{return} into partial fractions, which requires a study of the polynomials
$$
P_{\delta}\parentheses{\sigma} \= \sigma^{3} + A\sqrt{d}\sigma^{2} + \parentheses{\mu+\delta}\sigma +A\sqrt{d}\delta,
\qquad
\delta\geq 0 \;\; \parentheses{\text{since }D>d}.
$$
We show in Appendix \ref{ss:appendix_P_delta} that the three complex roots $\alpha\parentheses{\delta},\beta\parentheses{\delta},\gamma\parentheses{\delta}$ of $P_{\delta}$ are simple for almost all $\xi\in\mathbb{R}^{N-1}$. Hence, for these $\xi$,
\begin{equation}\label{Lambda-frownhat}
\frownhat{\Lambda}\parentheses{s,\xi,y} =
a \; \frac{e^{-\frac{\sigma}{\sqrt{d}}y}}{\sigma-\alpha}
+
b \; \frac{e^{-\frac{\sigma}{\sqrt{d}}y}}{\sigma-\beta}
+
c \; \frac{e^{-\frac{\sigma}{\sqrt{d}}y}}{\sigma-\gamma},
\end{equation}
where $a$, $b$ and $c$ have been defined in \eqref{def:abc}, and do not depend on the the $s$ variable. Expression \eqref{Lambda-frownhat} of $\, \frownhat{\Lambda}$ shows up a known form of $t$-Laplace transform --- see Lemma \ref{lem:inverse-Laplace-Fourier}, $(\ref{A1-Laplace-E_sur_sig_moins_b})$ --- evaluated at $\sigma^2=s+d\vertii{\xi}{}^{2}$, namely
\begin{multline*}
\frownhat{\Lambda}\parentheses{s,\xi,y} =
{\Lb}\Bigg[
t\mapsto
e^{-\frac{y^{2}}{4dt}}
\Bigg(
\frac{a+b+c}{\sqrt{\pi t}}
+ a\alpha \, \frac{\Erfc}{\Gamma}\Parentheses{\frac{-2\alpha\sqrt{d}\, t + y}{2\sqrt{dt}}}
+ b\beta \, \frac{\Erfc}{\Gamma}\Parentheses{\frac{-2\beta\sqrt{d}\, t + y}{2\sqrt{dt}}}\\
+ c\gamma \, \frac{\Erfc}{\Gamma}\Parentheses{\frac{-2\gamma\sqrt{d}\, t + y}{2\sqrt{dt}}}
\Bigg)
\Bigg]
\parentheses{s+d\vertii{\xi}{}^{2}}.
\end{multline*}
Observing that $a+b+c=0$, using Lemma \ref{lem:propri-Laplace-Fourier}, $(\ref{A2-Delay_th})$, and recalling the notation $\Phi_{\bullet}$ defined in statement of Theorem \ref{th:heat-eq-FR-space-SOLUTION}, see \eqref{def:phi}, we thus reach
\begin{equation}\label{presque_Lambda_frownhat}
\frownhat{\Lambda}\parentheses{s,\xi,y} =
{\Lb}\crochets{
t\mapsto 
e^{-\frac{y^{2}}{4dt}}
\Big[ a\alpha\Phi_{\alpha}+b\beta\Phi_{\beta}+c\gamma\Phi_{\gamma}\Big] \parentheses{t,\xi,y} \;
e^{-dt\vertii{\xi}{}^{2}}
}
\parentheses{s}.
\end{equation}
We know from Lemma \ref{lem:inverse-Laplace-Fourier}, $(\ref{A1-Fourier-Gaussienne})$, that $e^{-dt\vertii{\xi}{}^{2}}={\Fb}\crochets{\GRd\parentheses{t,\point}}\parentheses{\xi}$ --- notice that here we do mean $\GRd$, and not $\GRD$, see \eqref{heat-kernel}. Denoting 
$$
\Phi \= a\alpha\Phi_{\alpha}+b\beta\Phi_{\beta}+c\gamma\Phi_{\gamma},
$$
we also know --- see Lemma \ref{lem:propri-Laplace-Fourier}, $(\ref{A2-Fourier-inversion})$ --- that $\Phi\parentheses{t,\xi,y} = \frac{1}{\parentheses{2\pi}^{N-1}}{\Fb}\crochets{{\Fb}\crochets{\Phi\parentheses{t, - \, \point \, , y}}}\parentheses{\xi}$. Hence, from Lemma \ref{lem:propri-Laplace-Fourier}, $(\ref{A2-Fourier-convolee})$, we end up with
\begin{align}
 	\frownhat{\Lambda}\parentheses{s,\xi,y}
 	&\nonumber= {\Lb}{\Fb}\Bigg[\parentheses{t,x}\mapsto 
 	\frac{e^{-\frac{y^{2}}{4dt}}}{\parentheses{2\pi}^{N-1}}
 	\int_{\eta\in\mathbb{R}^{N-1}}^{}
 	\int_{\chi\in\mathbb{R}^{N-1}}^{}
 	\Phi \parentheses{t,\chi,y} \;
 	\GRd\parentheses{t,x-\eta} \;
	e^{i\chi\cdot\eta}
 	d\chi
 	d\eta
 	\Bigg]
 	\parentheses{s,\xi}\\
 	&\nonumber\hspace{-6.747mm} = {\Lb}{\Fb}\Bigg[\parentheses{t,x}\mapsto 
 	\frac{e^{-\frac{y^{2}}{4dt}}}{\parentheses{2\pi}^{N-1}}
 	\int_{\chi\in\mathbb{R}^{N-1}}^{}
 	\Phi \parentheses{t,\chi,y} \;
 	e^{i\chi\cdot x}
 	\int_{\eta\in\mathbb{R}^{N-1}}^{}
 	\GRd\parentheses{t,x-\eta} \;
	e^{-i\chi\cdot\parentheses{x-\eta}}
 	d\eta \;
 	d\chi
 	\Bigg]
 	\parentheses{s,\xi}\\
 	&\label{Lambda}\hspace{-6.747mm} = {\Lb}{\Fb}\Bigg[\parentheses{t,x}\mapsto 
 	\frac{e^{-\frac{y^{2}}{4dt}}}{\parentheses{2\pi}^{N-1}}
 	\int_{\chi\in\mathbb{R}^{N-1}}^{}
 	\Phi \parentheses{t,\chi,y} \;
 	e^{-dt\vertii{\chi}{}^{2}+i\chi\cdot x}
 	d\chi
 	\Bigg]
 	\parentheses{s,\xi}.
\end{align}
In other words, and as expected, we do have
$$
\frownhat{\Lambda}\parentheses{s,\xi,y}= {\Lb}{\Fb}\bigg[\parentheses{t,x}\mapsto
 	\Lambda\parentheses{t,x,y}
 	\bigg]
 	\parentheses{s,\xi},
$$
where $\Lambda=\Lambda(t,x,y)$ has been introduced in the statement of Theorem \ref{th:heat-eq-FR-space-SOLUTION}, see \eqref{def:Lambda}. Finally, our last move consists in using Lemma \ref{lem:propri-Laplace-Fourier}, ($\ref{A2-Fourier-convolee}$) and ($\ref{A2-Laplace-convolee}$),  in \eqref{v-frown-hat-explicit}, leading us, knowing $\Lambda$, to
\begin{multline*}
\frownhat{v}\parentheses{s,\xi,y} =
{\Lb}{\Fb}\Bigg[\parentheses{t,x}\mapsto
V\parentheses{t,x,y}
+ B\sqrt{d} \int_{\mathbb{R}^{N-1}}^{}
\Lambda\parentheses{t,z,y} \;
u_{0}\parentheses{x-z}
dz\\
+ \nu B\sqrt{d}
\int_{0}^{t}
\int_{\mathbb{R}^{N-1}}^{}
\Lambda\parentheses{\tau,z,y} \;
V|_{y=0}\parentheses{t-\tau , x-z}
dz
d\tau
\Bigg]
\parentheses{s,\xi}
\end{multline*}
which provides \eqref{v-sol-FR}.
\end{proof}

\begin{remark}[When $D\leq d$]\label{rem:D-egal-d} Let us first discuss the case $D<d$. Then, when $\xi$ browses $\R^{N-1}$, $\delta=\parentheses{D-d}\vertii{\xi}{}^{2}$ browses $\intervalleof{-\infty}{0}$. As in Appendix \ref{ss:appendix_P_delta}, it can be shown that, for almost all $\delta \leq 0$ (and thus for almost all $\xi \in \R^{N-1}$), the roots of $P_\delta$ are simple. We thus can write again \eqref{Lambda-frownhat} and the above proof readily applies.

We now turn on the critical case $D=d$. When $\xi$ browses $\R^{N-1}$, $\delta=\parentheses{D-d}\vertii{\xi}{}^{2}$ remains stuck  at zero --- so that $\Phi(t,\xi,y)$ is actually independent on $\xi$ --- and only  the polynomial $P_{0}$ plays a role. If $\mu \neq \frac{\nu^2}{4d}$, then $P_{0}$ has three simple roots (see Figure \ref{figure-Roots} and Appendix \ref{ss:appendix_P_delta}) and the above proof again readily applies. On the other hand, if $\mu=\frac{\nu^2}{4d}$, then $P_{0}$ has a simple root and a double root. In this case, one should go back to \eqref{return} and, rather than \eqref{Lambda-frownhat}, use the adequate expansion into partial fractions. Details are omitted.
\end{remark}

\MAT{
\begin{remark}[Adding linear growth]\label{rem:linear-growth} One may also want to consider the case with linear growth both in the field and on the road.  If, in the purely diffusive system \eqref{syst}, $+pv$ and $+qu$ are added to the $v$-equation and the $u$-equation respectively, then by considering
$(\widetilde v (t,x),\widetilde u(t,x)) \= e^{-qt}(v(t,x),u(t,x))$,
the system is recast (after dropping the tildes and defining $r\= p-q$)
$$
\left\{
\begin{array}{ll}
	\partial_{t} v = d \Delta v+rv,
&\quad t>0, \; x \in \R^{N-1}, \; y>0, \vspace{4pt}\\
	- d \partial_{y} v|_{y=0} = \mu u - \nu v|_{y=0},
&\quad t>0, \; x \in \R^{N-1}, \vspace{4pt}\\
	\partial_{t} u = D \Delta u + \nu v|_{y=0} - \mu u,
&\quad t>0, \; x \in \R^{N-1},
\end{array}
\right.
$$
which is the linearized system around the null steady state of the Fisher-KPP system \eqref{syst-non-lin} originally introduced in \cite{Ber-Roq-Ros-13-1, Ber-Roq-Ros-13-2}. Then the first change in the proof is that \eqref{4_ODE_sur_frownhat_v} is transferred into
\begin{equation}\label{4_ODE_sur_frownhat_v-AVEC-r}
d \partial_{yy} \frownhat{v}\parentheses{s,\xi,y}-
	\Parentheses{s+d\vertii{\xi}{}^{2}-r}\frownhat{v}\parentheses{s,\xi,y}
	=
	-\hat{\,v_{0}}\parentheses{\xi ,y},
\qquad y>0.
\end{equation}
Since individuals are expected to better grow in the field than on the road, we would typically have $r>0$ and, because of that, the structure of the solutions to the linear second order ODE \eqref{4_ODE_sur_frownhat_v-AVEC-r}  does depend upon parameters $s$ and $r$, and different cases should be considered. As a result our method readily applies but, to reach a rather explicit expression of the solution possibly yielding new insights, further heavy computations would be necessary.
\end{remark}
}

%%%%%%%%%%%%%%%%%%%%%%%%%%%%%%%%%%%%%%%%%%%%%%%%%%%%%%%%%%%%%%%%%%%%%%%%%%
%%%%%%%%%%%%%%%%%%%%%%%%%%%%%%%%%%%%%%%%%%%%%%%%%%%%%%%%%%%%%%%%%%%%%%%%%%
%%%%%%%%%%%%%%%%%%%%%%%%%%%%%%%%%%%%%%%%%%%%%%%%%%%%%%%%%%%%%%%%%%%%%%%%%%

\section{The decay rate of the field-road model}\label{s:decay}

In this section, we estimate the decay rate of the $L^{\infty}$ norm of the solution to the Cauchy problem \eqref{syst}\,--\,\eqref{initial_data-bis}. Owing to the parabolic comparison principle, we can assume without loss of generality that $v_0$ and $u_0$ are smooth. 

\begin{proof}[Proof of Theorem \ref{th:decay}]
We first handle the case of $v$ --- from which the one of $u$ shall easily ensue. Reminding the expression \eqref{v-sol-FR} of $v$,
\begin{multline}
v\parentheses{t,X} =
V\parentheses{t,X}
+
\frac{\mu}{\sqrt{d}}
\overbrace{\int_{\mathbb{R}^{N-1}}^{}
\Lambda\parentheses{t,z,y} \;
u_{0}\parentheses{x-z} \;
dz}^{\text{Let us call this $P\parentheses{t,x,y}$,}}\\
%%%
+
\frac{\mu \, \nu}{\sqrt{d}}
\underbrace{\int_{0}^{t}
\int_{\mathbb{R}^{N-1}}^{}
\Lambda\parentheses{s,z,y} \;
V|_{y=0}\parentheses{t-s,x-z} \;
dz \,
ds}_{\text{and that $Q\parentheses{t,x,y}$}.},\label{truc}
\end{multline}
The first term, namely $V\parentheses{t,X}$, is the solution to the Cauchy problem \eqref{eq_pour_V} whose $L^{\infty}$ control is given in statement of Theorem \ref{th:heat-eq-half-space} $(\ref{HE-Half-Space-Robin})$:\begin{equation}\label{controle_V}
\vertii{V(t,\point)}{{L^{\infty}(\mathbb{R}^{N}_{+})}}
\lesssim
	\frac  {\int_{\R_+^N}(1+\omega)\vert v_{0}(Z)\vert\, dZ}  {t^{\frac{N+1}{2}}},
\qquad \forall t>0.
\end{equation}
To deal with terms $P$ and $Q$, we will need the following lemma, whose proof is postponed at the end of this section.

\begin{lemma}[$L^{\infty}$ control of $\Phi$]\label{lem:controle-L-infty-Phi}
Recall the notation introduced in the proof of Theorem \ref{th:heat-eq-FR-space-SOLUTION}:
\begin{equation}
\label{recall-Phi}
\Phi \parentheses{t,\xi,y} \=
\Big[ a\alpha\Phi_{\alpha}+b\beta\Phi_{\beta}+c\gamma\Phi_{\gamma}\Big] \parentheses{t,\xi,y}.
\end{equation}
Then 
\begin{equation}\label{controle-L-infty-Phi}
{\Sb}\parentheses{t} \=
\;\;
\sup\limits_{y\geq 0}
\;\;
\sup\limits_{\xi\in\mathbb{R}^{N-1}}
\;\;
\verti{\Phi\parentheses{t,\xi,y}} \lesssim
\frac{1}{\sqrt{1+t}}, \qquad \forall t > 0.
\end{equation}
\end{lemma}

From here, the estimate of $P$ is straight. Indeed, using expression  \eqref{def:Lambda} of $\Lambda$ as written in Theorem \ref{th:heat-eq-FR-space-SOLUTION}, we get
\begin{align*}
 	\verti{P\parentheses{t,x,y}}
 	&= \verti{\int_{z\in\mathbb{R}^{N-1}}^{}\frac{e^{-\frac{y^{2}}{4dt}}}{\parentheses{2\pi}^{N-1}}
 	\int_{\xi\in\mathbb{R}^{N-1}}^{}
 	\Phi\parentheses{t,\xi,y} \;
 	e^{-dt\vertii{\xi}{}^{2}+i \xi\cdot z} \;
 	u_{0}\parentheses{x-z} \;
 	d\xi \,
 	dz}\\
 	&\leq \frac{{\Sb}\parentheses{t}}{\parentheses{2\pi}^{N-1}} \;
 	\int_{\xi\in\mathbb{R}^{N-1}}^{}e^{-dt\vertii{\xi}{}^{2}} \, d\xi \;
 	\int_{z\in\mathbb{R}^{N-1}}^{}\verti{u_{0}\parentheses{z}}dz\\
 	&= \frac{{\Sb}\parentheses{t}}{\parentheses{4\pi dt}^{\frac{N-1}{2}}} \;
 	\int_{z\in\mathbb{R}^{N-1}}^{}\verti{u_{0}\parentheses{z}}dz.
\end{align*}
Whence, thanks to Lemma \ref{lem:controle-L-infty-Phi},
\begin{equation}\label{controle_P}
\vertii{P(t,\point)}{{L^{\infty}(\mathbb{R}^{N}_{+})}}
\lesssim
	\frac{\int_{\mathbb{R}^{N-1}}^{}\verti{u_{0}\parentheses{z}}dz}{t^{\frac{N}{2}}},
\qquad \forall t>0.
\end{equation}

Controlling $Q$ is more refined. Observe from Theorem \ref{th:heat-eq-half-space} and Corollary \ref{cor:split} that
$$
V(t,x,y) =
\int_{\eta\in\mathbb{R}^{N-1}}^{}
\int_{\omega=0}^{\infty}
	\GRd(t,x-\eta) \;
	\HalfaONE (t,y,\omega) \;
	v_{0}(\eta,\omega) \;
d\omega \,
d\eta,
$$
where $\theta \= \frac{\nu}{1+\nu}\in (0,1)$, so that
\begin{align*}
Q\parentheses{t,x,y}
 	&= \int_{s=0}^{t}
 	\int_{z\in\mathbb{R}^{N-1}}^{}\frac{e^{-\frac{y^{2}}{4ds}}}{\parentheses{2\pi}^{N-1}}
 	\int_{\xi\in\mathbb{R}^{N-1}}^{}
 	\Phi\parentheses{s,\xi,y} \;
 	e^{-ds\vertii{\xi}{}^{2}+i \xi\cdot z} \\
 	&\hspace{20.477mm} \int_{\eta\in\mathbb{R}^{N-1}}^{}\int_{\omega=0}^{\infty} 
 	\GRd\parentheses{t-s,x-z-\eta} \;
 	\HalfaONE\parentheses{t-s,0,\omega} \;
 	v_{0}\parentheses{\eta,\omega} \;
 	d\omega \,
 	d\eta \,
 	d\xi \,
 	dz \,
 	ds\\
%%%%
 	&= \int_{s=0}^{t}
 	\frac{e^{-\frac{y^{2}}{4ds}}}{\parentheses{2\pi}^{N-1}}
 	\int_{\xi\in\mathbb{R}^{N-1}}^{}
 	\Phi\parentheses{s,\xi,y} \;
 	e^{-ds\vertii{\xi}{}^{2}}
 	\int_{\eta\in\mathbb{R}^{N-1}}^{}\int_{\omega=0}^{\infty} 
 	\HalfaONE\parentheses{t-s,0,\omega} \;
 	v_{0}\parentheses{\eta,\omega} \\
 	&\hspace{62.375mm} \int_{z\in\mathbb{R}^{N-1}}^{}
 	\GRd\parentheses{t-s,x-z-\eta} \;
 	e^{i \xi\cdot z}
 	dz \,
 	d\omega \,
 	d\eta \,
 	d\xi \,
 	ds.\\
 	%%%%
 	\intertext{The integral in $z$ is nothing else than ${\Fb}\crochets{\GRd\parentheses{t-s,\point+x-\eta}}\parentheses{\xi}=e^{-d\parentheses{t-s}\vertii{\xi}{}^{2}+i\xi\cdot \parentheses{x-\eta}}$ so that}
 	%%%%
Q\parentheses{t,x,y}
 	&= \int_{s=0}^{t}
 	\frac{e^{-\frac{y^{2}}{4ds}}}{\parentheses{2\pi}^{N-1}}
 	\int_{\xi\in\mathbb{R}^{N-1}}^{}
 	\Phi\parentheses{s,\xi,y} \;
 	e^{-dt\vertii{\xi}{}^{2}}
 	\int_{\eta\in\mathbb{R}^{N-1}}^{}\int_{\omega=0}^{\infty} 
 	\HalfaONE\parentheses{t-s,0,\omega} \;
 	v_{0}\parentheses{\eta,\omega} \\
 	&\hspace{108.266mm} e^{i\xi\cdot \parentheses{x-\eta}} \,
 	d\omega \,
 	d\eta \,
 	d\xi \,
 	ds\\
 	&=\int_{s=0}^{t}
 	\frac{e^{-\frac{y^{2}}{4ds}}}{\parentheses{2\pi}^{N-1}}
 	\int_{\xi\in\mathbb{R}^{N-1}}^{}
 	\Phi\parentheses{s,\xi,y} \;
 	e^{-dt\vertii{\xi}{}^{2}} e^{i\xi\cdot x}
 	\int_{\omega=0}^{\infty} 
 	\HalfaONE\parentheses{t-s,0,\omega} \;
 	\hat{v_{0}}\parentheses{\xi,\omega} \,
 	d\omega \,
 	d\xi \,
 	ds.
\end{align*}
As a result,
$$
\verti{Q\parentheses{t,x,y}}
\leq
\frac{1}{\parentheses{2\pi}^{N-1}}
\int_{s=0}^{t}
{\Sb}\parentheses{s}
\int_{\xi\in\mathbb{R}^{N-1}}^{}
e^{-dt\vertii{\xi}{}^{2}}
 \int_{\omega=0}^{\infty} 
\HalfaONE\parentheses{t-s,0,\omega}
\verti{\?\?\?\?\hat{v_{0}}\parentheses{\xi,\omega}}
d\omega \,
d\xi \,
ds.
$$
From the control on ${\Sb}$ in Lemma \ref{lem:controle-L-infty-Phi} and estimate \eqref{decay-robin-bis} in Theorem \ref{th:heat-eq-half-space} (with $N=1$), we reach
$$
\verti{Q\parentheses{t,x,y}}
\lesssim
\TCvzero
\int_{s=0}^{t} \int_{\xi\in\mathbb{R}^{N-1}}^{}
e^{-dt\vertii{\xi}{}^{2}}
\frac{1}{\sqrt{1+s}\parentheses{1+t-s}}
d\xi\, ds
\lessim \frac{{\TCvzero}}{t^{\frac{N-1}{2}}} \int_{0}^{t} 
\frac{1}{\sqrt{1+s}\parentheses{1+t-s}}
ds,
$$
for some $\TCvzero\geq 0$. 
Finally, an elementary computation shows
$$
\int_{0}^{t}\frac{1}{\sqrt{1+s}\parentheses{1+t-s}}ds =
\frac{1}{\sqrt{2+t}} \crochets{\ln\Parentheses{\frac{1+\sqrt{\frac{1+t}{2+t}}}{1-\sqrt{\frac{1+t}{2+t}}}} - \ln\Parentheses{\frac{1+\sqrt{\frac{1}{2+t}}}{1-\sqrt{\frac{1}{2+t}}}}}
\leq \frac{\ln\parentheses{4t+6}}{\sqrt{2+t}},
$$
from which there comes
\begin{equation}\label{controle_Q}
\vertii{Q(t,\point)}{{L^{\infty}(\mathbb{R}^{N}_{+})}}
\lesssim \TCvzero
	\frac{\ln\parentheses{1+t}}{t^{\frac{N}{2}}},
\qquad \forall t>0.
\end{equation}
Gathering \eqref{controle_V}, \eqref{controle_P} and \eqref{controle_Q} brings us to the announced control \eqref{decay-v} on $v$.

\medskip

Let us now turn to $u$. We recall its expression as written in \eqref{u-sol-FR}:
$$
u\parentheses{t,x} =
e^{-\mu t} U\parentheses{t,x}
+
\nu \int_{0}^{t}
	e^{-\mu\parentheses{t-s}}
	\int_{\mathbb{R}^{N-1}}^{}
	\GRD\parentheses{t-s , x-z} \;
	v|_{y=0}\parentheses{s,z} \;
	dz \,
ds,
$$
where $U=U\parentheses{t,x}$ denotes the well-known solution to the Cauchy problem \eqref{eq_pour_U}. Hence, thanks the control \eqref{decay-v} we just established on $v$,
\begin{equation}\label{controle_u_avec_convolee_en_s}
\verti{u\parentheses{t,x}}
\lesssim
\frac{\Vert u_0\Vert_{L^1(\R^{N-1})} \, e^{-\mu t}}{t^{\frac{N-1}{2}}}+  \Cvzero \, e^{-\mu t}
 \int_{0}^{t} e^{\mu s} \, \frac{\ln\parentheses{1+s}}{\parentheses{1+s}^{\frac{N}{2}}} \; ds
+ \Cvuzero \, e^{-\mu t}
\int_{0}^{t} e^{\mu s} \, \frac{1}{\parentheses{1+s}^{\frac{N}{2}}} \; ds.
\end{equation}
We now  treat the two temporal convolutions in \eqref{controle_u_avec_convolee_en_s}. To do so, remark that
$$
\frac{d}{ds}\Parentheses{\frac{e^{\mu s}}{\mu}\,\frac{\ln\parentheses{1+s}}{\parentheses{1+s}^{\frac{N}{2}}}}
=
e^{\mu s} \, \frac{\ln\parentheses{1+s}}{\parentheses{1+s}^{\frac{N}{2}}} +
\frac{e^{\mu s}}{\mu} \, \frac{1-\frac{N}{2}\ln\parentheses{1+s}}{\parentheses{1+s}^{\frac{N+2}{2}}}
\stackrel{s\to\infty}{\sim}
e^{\mu s} \, \frac{\ln\parentheses{1+s}}{\parentheses{1+s}^{\frac{N}{2}}}$$
yields $\int_{0}^{t} e^{\mu s} \frac{\ln\parentheses{1+s}}{\parentheses{1+s}^{\frac{N}{2}}} \, ds \stackrel{t\to\infty}{\sim} \frac{e^{\mu t}}{\mu} \, \frac{\ln\parentheses{1+t}}{\parentheses{1+t}^{\frac{N}{2}}}$. Similarly,
$$
\frac{d}{ds}\Parentheses{\frac{e^{\mu s}}{\mu}\,\frac{1}{\parentheses{1+s}^{\frac{N}{2}}}}
=
e^{\mu s} \, \frac{1}{\parentheses{1+s}^{\frac{N}{2}}} -
\frac{e^{\mu s}}{\mu} \, \frac{\frac{N}{2}}{\parentheses{1+s}^{\frac{N+2}{2}}}
\stackrel{s\to\infty}{\sim}
e^{\mu s} \, \frac{1}{\parentheses{1+s}^{\frac{N}{2}}}
$$
yields $\int_{0}^{t} e^{\mu s} \frac{1}{\parentheses{1+s}^{\frac{N}{2}}} \, ds \stackrel{t\to\infty}{\sim} \frac{e^{\mu t}}{\mu} \, \frac{1}{\parentheses{1+t}^{\frac{N}{2}}} $. Therefore,
\begin{align*}
\verti{u\parentheses{t,x}}
	&\lesssim
	\frac{\Vert u_0\Vert_{L^1(\R^{N-1})} \, e^{-\mu t}}{t^{\frac{N-1}{2}}}+
	\frac{\Cvzero \ln\parentheses{1+t}}{\parentheses{1+t}^{\frac{N}{2}}} + \frac{\Cvuzero}{\parentheses{1+t}^{\frac{N}{2}}},
\end{align*}
which provides the control \eqref{decay-u} on $u$ up to changing the value of $\Cvuzero$ if necessary.
\end{proof}

It now remains  to prove Lemma \ref{lem:controle-L-infty-Phi}.

\begin{proof}[Proof of Lemma \ref{lem:controle-L-infty-Phi}]  The proof is related to the study of the $\delta$-indexed polynomials $P_{\delta}$  in Appendix \ref{ss:appendix_P_delta}. As $\delta = \parentheses{D-d}\vertii{\xi}{}^{2}$ browses $\mathbb{R}_{+}$, the polynomials $P_{\delta}$ may provide double or triple roots for, at most, two values of $\delta$ --- see Lemmas \ref{lem:fact_P_delta_only_simple_roots} to \ref{lem:fact_P_delta_1-double-root} $(\ref{ASimpleRoots-kind-of-roots})$. In Lemmas \ref{lem:fact_P_delta_only_simple_roots} to \ref{lem:fact_P_delta_1-double-root} we build a closed set ${\Eb}$ of $\R_+$ so that, for $\delta\in{\Eb}$, the roots $\alpha$, $\beta$, $\gamma$ remain away from each other and so that each connected component of ${\Eb}^{c}$ contains \textit{one and only one} $\delta_{i}$ ($i=0$, $1$ or $2$) which provides a multiple root $\lambda_{i}$. This proof is divided into three steps:
\begin{itemize}
	\item[($i$)] We treat the cases where $\delta\in{\Eb}$ for which, thanks to the boundedness of $a$, $b$ and $c$, we can consider $a\alpha\Phi_{\alpha}$, $b\beta\Phi_{\beta}$ and $c\gamma\Phi_{\gamma}$ independently.
	
We deal then with the situations where $\delta\in{\Eb}^{c}$. In those, because $a$, $b$ and $c$ may be unbounded we must consider $a\alpha\Phi_{\alpha}$, $b\beta\Phi_{\beta}$ and $c\gamma\Phi_{\gamma}$ together. Let $\delta\in{\Eb}^{c}$, then there is $i\in\ensemble{0,1,2}$ such that $\delta\in\intervalleoo{\delta_{i}-\eta}{\delta_{i}+\eta}$ --- recall that $\delta=\delta_{i}$ provides the multiple root $\lambda_{i}$ and that $\eta$ is defined in Lemmas \ref{lem:fact_P_delta_triple-root} to \ref{lem:fact_P_delta_1-double-root}.
	\item[($ii$)] We treat the cases where the root $\lambda_{i}$ is double.	\item[($iii$)] We treat the cases where the root $\lambda_{i}$ is triple.
\end{itemize}
Before starting, we emphasize that \eqref{partie-reelle-neg} insures
$
\realpart{\alpha}, \realpart{\beta}, \realpart{\gamma} \leq 0
$
which allows us to use the estimates \eqref{Control_ERFC} at point $z = \frac{-2 \, \bullet \, \sqrt{d}t+y}{2\sqrt{dt}}$ for $\bullet$ in the convex hull of $\alpha$, $\beta$ and $\gamma$. Notice that, in the sequel, we shall often write $\Phi_\bullet$ for $\Phi _\bullet (t,\xi,y)$.

\medskip
Step ($i$). Because $\inf\limits_{\delta\in{\Eb}} \Parentheses{\verti{\alpha-\beta} , \verti{\alpha-\gamma} , \verti{\beta-\gamma}} > \varepsilon$, we have
\begin{equation}\label{a_b_c_bounded}
\sup\limits_{\delta\in{\Eb}} \Parentheses{\verti{a} , \verti{b} , \verti{c}} < \frac{1}{\varepsilon^{2}}.
\end{equation}
Next, using  \eqref{Control_ERFC}, we have, for $\lambda\in\ensemble{\alpha,\beta,\gamma}$, 
$$
 	\verti{\lambda\Phi_{\lambda}}
 	= \verti{\lambda \frac{\Erfc}{\Gamma}\Parentheses{\frac{-2\lambda\sqrt{d}t+y}{2\sqrt{dt}}}}
 	\lesssim \frac{2\verti{\lambda}\sqrt{dt}}{\verti{-2\lambda\sqrt{d}t+y}} 	\leq \frac{1}{\sqrt{t}},
$$
the last inequality coming from $\realpart{\lambda} \leq 0$ and $y\geq 0$. 
Gathering the latter inequality with \eqref{a_b_c_bounded} provides, in view of \eqref{recall-Phi},
\begin{equation}\label{Controle_preuve_du_lemme_1}
\sup\limits_{y\geq 0}
\;\;
\sup\limits_{\delta\in{\Eb}}
\;\;
\verti{\Phi\parentheses{t,\xi,y}}\lesssim\frac{1}{\sqrt{1+t}} \quad  \text{ for all } t\geq 1.
\end{equation}
We now turn to the case where $0<t\leq 1$ (still when  $\delta\in{\Eb}$). Let $\delta_{\infty}>0$ be  defined as in Lemma \ref{lem:fact_P_delta_large_values_of_delta}. It is readily seen that  
\begin{equation}\label{delta-compact}
\sup\limits_{t\in\intervalleof{0}{1}}
\;\;
\sup\limits_{y\geq 0}
\;\;
\sup\limits_{\delta\in{\Eb}\cap\intervalleff{0}{\delta_{\infty}}}
\;\;
\verti{\Phi\parentheses{t,\xi,y}}<+\infty.
\end{equation}
Next let us consider $\delta\geq\delta_{\infty}$. Since $\gamma=\overline \beta$, $c=\overline b$, $\Phi_\gamma=\overline{\Phi_\beta}$, we deduce from \eqref{recall-Phi} that
\begin{equation}\label{Controle_du_Phi_t_leq_1}
\verti{\Phi}\leq
\verti{a\alpha\Phi_{\alpha}} + 2 \verti{b\beta\Phi_{\beta}}.
\end{equation}
The first term in the above right-hand-side is clearly bounded while, for the second,
$$
\verti{b\beta\Phi_{\beta}} = \frac{\verti{\beta}\verti{\Phi_{\beta}}}{\verti{\beta-\alpha}\verti{\beta-\gamma}} \leq \frac{\verti{\Phi_{\beta}}}{\verti{\beta-\alpha}} \leq \frac{\verti{\Phi_{\beta}}}{\varepsilon} <+\infty.
$$
Hence
\begin{equation}\label{delta-gd}
\sup\limits_{t\in\intervalleof{0}{1}}
\;\;
\sup\limits_{y\geq 0}
\;\;
\sup\limits_{\delta\in{\Eb}\cap[\delta_{\infty},+\infty)}
\;\;
\verti{\Phi\parentheses{t,\xi,y}}<+\infty.
\end{equation}
From \eqref{delta-compact} and \eqref{delta-gd}, we deduce that \eqref{Controle_preuve_du_lemme_1} actually holds true for all $t > 0$,  and we are done with this case.

Step ($ii$). Without loss of generality, we may, to fix ideas, suppose that $\delta\to\delta_{i}$ ($i=1\text{ or }2$) provides the merging of the simple roots $\alpha$ and $\beta$ into the double root $\lambda_{i}$ --- see Lemmas \ref{lem:fact_P_delta_2-double-roots} and \ref{lem:fact_P_delta_1-double-root}. Recall in that case that
$$
\inf\limits_{\delta\in\intervalleoo{\delta_{i}-\eta}{\delta_{i}+\eta}}
\;\;
\inf\limits_{\lambda\in\intervalleff{\alpha}{\beta}}
\;\;
\parentheses{\verti{\lambda},\verti{\lambda-\gamma}} > \varepsilon,
$$
for some $\eta>0$. Then, for $\delta\in\intervalleoo{\delta_{i}-\eta}{\delta_{i}+\eta}$ and $\lambda\in\intervalleff{\alpha}{\beta}$, one sets
$$
\psi_{\lambda} = \psi_{\lambda}\parentheses{t,\xi,y} \= \frac{\lambda\parentheses{\Phi_{\lambda}-\Phi_{\gamma}}}{\lambda-\gamma},
$$
and we claim that
\begin{equation}\label{Controle_psi_prime}
\sup\limits_{y\geq 0}
\;\;
\sup\limits_{\delta\in\intervalleoo{\delta_{i}-\eta}{\delta_{i}+\eta}} 
\;\;
\sup\limits_{\lambda \in \intervalleff{\alpha}{\beta}} 
\;\;
\verti{\psi_{\lambda}'}\lesssim\frac{1}{\sqrt{1+t}},
\qquad \forall t > 0,
\end{equation}
where $\psi'_{\lambda}$ denotes $\parentheses{\partial_{\bullet}\psi_{\bullet}}_{\lambda}\parentheses{t,\xi,y}$. Using $a\alpha+b\beta+c\gamma=0$ we get
\begin{align*}
 	\verti{\Phi}
	&= \verti{a\alpha\parentheses{\Phi_{\alpha}-\Phi_{\gamma}}+b\beta\parentheses{\Phi_{\beta}-\Phi_{\gamma}}}\\
 	&= \frac{\verti{\psi_{\alpha}-\psi_{\beta}}}{\verti{\alpha-\beta}}\\
 	&\leq \; \sup\limits_{y\geq 0}\;\;\sup\limits_{\delta\in\intervalleoo{\delta_{i}-\eta}{\delta_{i}+\eta}} \;\;
\sup\limits_{\lambda \in \intervalleff{\alpha}{\beta}} \;\;\verti{\psi_{\lambda}'}\\
 	&\lesssim \frac{1}{\sqrt{1+t}},
\end{align*}
where we used the mean value inequality to provide third line.

It thus only remains to prove the claim \eqref{Controle_psi_prime}. We have
$$
\psi'_{\lambda} = \frac{\lambda\Phi'_{\lambda}}{\lambda-\gamma}
- \frac{\gamma\Phi_{\lambda}}{\parentheses{\lambda-\gamma}^{2}}
+\frac{\gamma\Phi_{\gamma}}{\parentheses{\lambda-\gamma}^{2}},
$$
where $\Phi'_{\lambda}$ denotes $\parentheses{\partial_{\bullet}\Phi_{\bullet}}_{\lambda}\parentheses{t,\xi,y}$. It is at first clear that
\begin{equation}\label{proof_lemma_small_t_case_2}
\sup\limits_{t \in \intervalleof{0}{1}}
\;\; 
\sup\limits_{y\geq 0}
\;\;
\sup\limits_{\delta\in\intervalleoo{\delta_{i}-\eta}{\delta_{i}+\eta}} 
\;\;
\sup\limits_{\lambda \in \intervalleff{\alpha}{\beta}} 
\;\;
\verti{\psi'_{\lambda}}
< +\infty.
\end{equation}
Next, using control \eqref{Control_ERFC}, we have, for $\lambda\in\intervalleff{\alpha}{\beta}$,
$$
\verti{\Phi'_{\lambda}} =
\sqrt{t}\verti{\Parentheses{\frac{\Erfc}{\Gamma}}'\!\Parentheses{\frac{-2\lambda\sqrt{d}t+y}{2\sqrt{dt}}}
}\lesssim \sqrt{t} \; \frac{4dt}{\verti{-2\lambda\sqrt{d}t+y}^{2}}
\leq \frac{1}{\verti{\lambda}^{2}\sqrt{t}},
$$
$$
 	\verti{\Phi_{\lambda}}
 	= \verti{\frac{\Erfc}{\Gamma}\Parentheses{\frac{-2\lambda\sqrt{d}t+y}{2\sqrt{dt}}}}
 	\lesssim \frac{2\sqrt{dt}}{\verti{-2\lambda\sqrt{d}t+y}}
 	\leq \frac{1}{\verti{\lambda}\sqrt{t}},
$$
$$
 	\verti{\gamma\Phi_{\gamma}}
 	= \verti{\gamma \frac{\Erfc}{\Gamma}\Parentheses{\frac{-2\gamma\sqrt{d}t+y}{2\sqrt{dt}}}}
 	\lesssim \frac{\verti{\gamma}2\sqrt{dt}}{\verti{-2\gamma\sqrt{d}t+y}}
 	\leq \frac{1}{\sqrt{t}},
$$
and so, because $\verti{\lambda}>\varepsilon$ --- see Lemmas \ref{lem:fact_P_delta_2-double-roots} and \ref{lem:fact_P_delta_1-double-root}, ($\ref{AOneDoubleRoots-seg-away-from-zero-and-gamma}$) ---, one gets, for all $t\geq 1$,
\begin{equation}\label{proof_lemma_large_t_case_2}
\sup\limits_{y\geq 0}
\;\; 
\sup\limits_{\delta\in\intervalleoo{\delta_{i}-\eta}{\delta_{i}+\eta}} 
\;\;
\sup\limits_{\lambda \in \intervalleff{\alpha}{\beta}} 
\;\;
\verti{\psi'_{\lambda}}
\lesssim \frac{1}{\sqrt{1+t}}.
\end{equation}
Gathering \eqref{proof_lemma_small_t_case_2} and \eqref{proof_lemma_large_t_case_2} finally gives \eqref{Controle_psi_prime} for all $t > 0$ and we are done with this case.

Step ($iii$). We recall in this last case that $\delta\to\delta_{0}$ provokes the merging of the three simple roots $\alpha$, $\beta$, $\gamma$ into the triple root $\lambda_{0} = -\frac{A\sqrt{d}}{3}$. Without loss of generality it may be assumed that $\alpha$ is the real root, and taking $\epsR$ and $\epsI$ as defined in Lemma \ref{lem:fact_P_delta_triple-root}, ($\ref{ATripleRoots-epsR-espI}$), we have
$$
\alpha = \lambda_{0} - 2\epsR,
\qquad
\beta = \lambda_{0} + \epsR + i\epsI,
\qquad
\gamma = \lambda_{0} + \epsR - i\epsI.
$$
Using again $a\alpha+b\beta+c\gamma=0$ yields
\begin{align}
 	\Phi
	&= b\beta\parentheses{\Phi_{\beta}-\Phi_{\alpha}} + c\gamma\parentheses{\Phi_{\gamma}-\Phi_{\alpha}} \nonumber\\
 	&= \frac{\beta}{\beta-\gamma}\frac{\Phi_{\beta}-\Phi_{\alpha}}{\beta-\alpha} - \frac{\gamma}{\beta-\gamma}\frac{\Phi_{\gamma}-\Phi_{\alpha}}{\gamma-\alpha} \nonumber\\
 	&= \frac{\beta}{\beta-\gamma}\frac{\Phi_{\beta}-\Phi_{\alpha}}{\beta-\alpha} - \frac{\gamma}{\beta-\gamma}\overline{\Parentheses{\frac{\Phi_{\beta}-\Phi_{\alpha}}{\beta-\alpha}}} \nonumber\\
 	&= \Realpart{\frac{\Phi_{\beta}-\Phi_{\alpha}}{\beta-\alpha}} + i\frac{\beta+\gamma}{\beta-\gamma}\Impart{\frac{\Phi_{\beta}-\Phi_{\alpha}}{\beta-\alpha}} \nonumber\\
 	&= \label{Control_Phi_case_2}\Realpart{\frac{\Phi_{\beta}-\Phi_{\alpha}}{\beta-\alpha}} + \frac{\lambda_{0}+\epsR}{\epsI}\Impart{\frac{\Phi_{\beta}-\Phi_{\alpha}}{\beta-\alpha}}.
\end{align}
The main issue is now to control second term in \eqref{Control_Phi_case_2} because one has to compensate the term $\epsI$ which vanishes as $\delta\to \delta_{0}$. To do so, we use a Taylor-Lagrange expansion of $\Phi_{\bullet}$ at $\bullet=\alpha$ which provides, for some $\lambda\in\intervalleff{\alpha}{\beta}$,
$$
\Phi_{\beta}=\Phi_{\alpha}+\parentheses{\beta-\alpha}\Phi'_{\alpha}+\frac{\parentheses{\beta-\alpha}^{2}}{2}\Phi''_{\lambda},
$$
where $\Phi'_{\alpha}$ denotes $\parentheses{\partial_{\bullet}\Phi_{\bullet}}_{\alpha}$, and $\Phi''_{\lambda}$ denotes $\parentheses{\partial_{\bullet\bullet}\Phi_{\bullet}}_{\lambda}$. Hence, one gets
$$
\frac{\Phi_{\beta}-\Phi_{\alpha}}{\beta-\alpha} =
\Phi'_{\alpha} + \frac{3\epsR + i \epsI}{2}\Phi''_{\lambda}.
$$
Now, because $\alpha$ is real,
$$
\Impart{\frac{\Phi_{\beta}-\Phi_{\alpha}}{\beta-\alpha}} =
\frac{3\epsR}{2}\impart{\Phi''_{\lambda}} +
\frac{\epsI}{2}\realpart{\Phi''_{\lambda}},
$$
and therefore, using that $\epsR\stackrel{\delta\to\delta_{0}}{=}\mathcal{O}\parentheses{\epsI}$,
$$
\verti{\Impart{\frac{\Phi_{\beta}-\Phi_{\alpha}}{\beta-\alpha}}} \leq 
\Parentheses{\frac{3}{2}\verti{\epsR}+\frac{1}{2}\verti{\epsI}}\verti{\Phi''_{\lambda}} \lesssim
\verti{\epsI \Phi''_{\lambda}}.
$$
As a consequence, by using the mean value inequality to control the first term in \eqref{Control_Phi_case_2}, one gets
$$
\verti{\Phi} \lesssim \; \sup\limits_{y\geq 0}\;\;\sup\limits_{\delta\in\intervalleoo{\delta_{0}-\eta}{\delta_{0}+\eta}} \;\;
\sup\limits_{\ell \in \intervalleff{\alpha}{\beta}} \;\; \verti{\Phi'_{\ell}} \;\; + \;\; \sup\limits_{y\geq 0}\;\;\sup\limits_{\delta\in\intervalleoo{\delta_{0}-\eta}{\delta_{0}+\eta}} \;\;
\sup\limits_{\lambda \in \intervalleff{\alpha}{\beta}} \;\; \verti{\lambda_{0}+\epsR}\verti{\Phi''_{\lambda}}.
$$
It is first clear that
\begin{equation}\label{proof_lemma_small_t_case_3}
\sup\limits_{t \in \intervalleof{0}{1}}
\;\; 
\sup\limits_{y\geq 0}
\;\;
\sup\limits_{\delta\in\intervalleoo{\delta_{0}-\eta}{\delta_{0}+\eta}} 
\;\;
\sup\limits_{\ell \in \intervalleff{\alpha}{\beta}} 
\;\;
\verti{\Phi'_{\ell}}
< +\infty,
\end{equation}
and
\begin{equation}\label{proof_lemma_small_t_case_3_bis}
\sup\limits_{t \in \intervalleof{0}{1}}
\;\; 
\sup\limits_{y\geq 0}
\;\;
\sup\limits_{\delta\in\intervalleoo{\delta_{0}-\eta}{\delta_{0}+\eta}} 
\;\;
\sup\limits_{\lambda \in \intervalleff{\alpha}{\beta}} 
\;\;
\verti{\Phi''_{\lambda}}
< +\infty.
\end{equation}
Now, using control \eqref{Control_ERFC}, we have, for $\ell \text{ and }\lambda\in\intervalleff{\alpha}{\beta}$,
$$
\verti{\Phi'_{\ell}} =
\sqrt{t}\verti{\Parentheses{\frac{\Erfc}{\Gamma}}'\!\Parentheses{\frac{-2\ell\sqrt{d}t+y}{2\sqrt{dt}}}
}\lesssim \sqrt{t} \; \frac{4dt}{\verti{-2\ell\sqrt{d}t+y}^{2}}
\leq \frac{1}{\verti{\ell}^{2}\!\sqrt{t}},
$$
$$
\verti{\Phi''_{\lambda}} =
t\verti{\Parentheses{\frac{\Erfc}{\Gamma}}''\!\Parentheses{\frac{-2\lambda\sqrt{d}t+y}{2\sqrt{dt}}}
}\lesssim t \; \frac{8\parentheses{dt}^{\frac{3}{2}}}{\verti{-2\lambda\sqrt{d}t+y}^{3}}
\leq \frac{1}{\verti{\lambda}^{3}\!\sqrt{t}},
$$
and thus, because $\verti{\ell} \text{ and }\verti{\lambda}>\varepsilon$ --- see Lemma \ref{lem:fact_P_delta_triple-root}, ($\ref{ATripleRoots-seg_away_from_zero}$) ---, one gets, for all $t\geq 1$,
\begin{equation}\label{proof_lemma_large_t_case_3}
\sup\limits_{y\geq 0}
\;\; 
\sup\limits_{\delta\in\intervalleoo{\delta_{0}-\eta}{\delta_{0}+\eta}} 
\;\;
\sup\limits_{\ell \in \intervalleff{\alpha}{\beta}} 
\;\;
\verti{\Phi'_{\ell}}
\lesssim \frac{1}{\sqrt{1+t}},
\end{equation}
\begin{equation}\label{proof_lemma_large_t_case_3_bis}
\sup\limits_{y\geq 0}
\;\; 
\sup\limits_{\delta\in\intervalleoo{\delta_{0}-\eta}{\delta_{0}+\eta}} 
\;\;
\sup\limits_{\lambda \in \intervalleff{\alpha}{\beta}} 
\;\;
\verti{\Phi''_{\lambda}}
\lesssim \frac{1}{\sqrt{1+t}}.
\end{equation}
Gathering \eqref{proof_lemma_small_t_case_3}, \eqref{proof_lemma_small_t_case_3_bis}, \eqref{proof_lemma_large_t_case_3} and \eqref{proof_lemma_large_t_case_3_bis} finally provides $\verti{\Phi}\lesssim \frac{1}{\sqrt{1+t}}$ for all $t > 0$, and we are done with this case which concludes the proof of the lemma.
\end{proof}

\vspace*{\fill}

%%%%%%%%%%%%%%%%%%%%%%%%%%%%%%%%%%%%%%%%%%%%%%%%%%%%%%%%%%%%%%%%%%%%%%%%%%
%%%%%%%%%%%%%%%%%%%%%%%%%%%%%%%%%%%%%%%%%%%%%%%%%%%%%%%%%%%%%%%%%%%%%%%%%%
%%%%%%%%%%%%%%%%%%%%%%%%%%%%%%%%%%%%%%%%%%%%%%%%%%%%%%%%%%%%%%%%%%%%%%%%%%

\section{Numerical explorations}\label{s:num}

Through this section, we assume $N=2$, meaning that the road is a line and we perform some numerical simulations of the Cauchy problem \eqref{syst}\,--\,\eqref{initial_data-bis}.  To do so, we truncate the unbounded domain $\R\times (0,+\infty)$ and work on a  bounded box $(-2M,2M)\times (0,M)$, where $M>0$. To preserve the quantity of individuals, we impose the no-flux boundary conditions through the artificial frontiers. The resolution method is  based on a classical finite difference scheme. Since we consider $M\gg 1$ 
and an initial datum $\parentheses{v_{0},u_{0}}$ supported in a \lq\lq relatively small'' compact compared to the size of the box, we are confident that the numerical solution should be close to the real solution. 

\medskip

\noindent {\bf On the decay rate.} When $D>d$, the asymptotic decay rate of the $L^\infty$ norm of the solution $(v,u)$ is expected to be of the magnitude $\mathcal{O}((1+t)^{-1})$, see Theorem \ref{th:decay} and the discussion on the logarithmic term right after. This is numerically confirmed by the left panel of Figure \ref{figure-Decay_Rate}.

On the other hand, when $D\leq d$ (which is not the essence of the field-road model), one should, based on Remark \ref{rem:D-egal-d},  redo the lengthy computations of Section \ref{s:decay} to capture the asympotic decay rate. We are rather confident that this should not alter the result, and  this is sustained by numerical simulations, see the right panel of Figure \ref{figure-Decay_Rate}.

\refstepcounter{FIGURE}\label{figure-Decay_Rate}
\begin{center}
\includegraphics[scale=1]{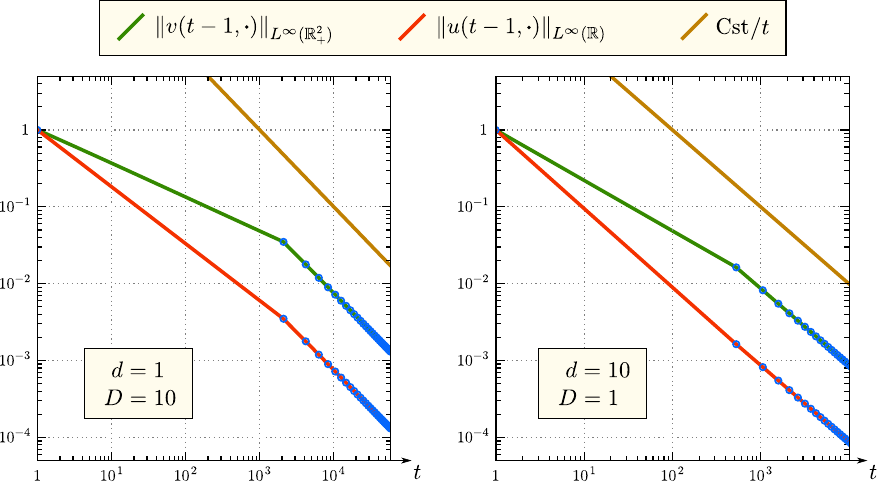}\\
\textsl{\textbf{Figure \theFIGURE} --- Decay in $\log$-scale of the solution to the Cauchy problem \eqref{syst}\?\?\?--\,\eqref{initial_data-bis} with $N=2$, $M=800$, $\mu=10$ and $\nu=1$, arising from the datum $\parentheses{v_{0},u_{0}}\equiv \parentheses{\indicatrice{\intervalleff{-10}{10}\times\intervalleff{10}{30}},\indicatrice{\intervalleff{-10}{10}}}$.}
\end{center}

\medskip

\noindent {\bf On the level sets.} We now turn to numerically explore the form of the level sets of the solution $(v,u)$. The results are presented in Figure \ref{figure-Solution2}.

In particular on the bottom panel of Figure \ref{figure-Solution2}, corresponding to a situation where $D=100$ and $d=1$, we observe the following. There holds
$$
-\partial _y v|_{y=0}(t=1000,\point)<0 \quad  \text{ in the \lq\lq middle of the road''},
$$
meaning that the individuals mainly switch {\it from} the field {\it to} the road. On the other hand,  there holds
$$
-\partial_y v|_{y=0}(t=1000,\point)>0 \quad  \text{ \lq\lq far away'' on the road},
$$
meaning that the individuals mainly switch {\it from} the road {\it to} the field. Roughly speaking, the road sucks up individuals in the central region (corresponding to the bulk of the population) and spits them out in the far away region (corresponding to the tails of the population). 

\MAT{
On the other hand, on the top panel of Figure \ref{figure-Solution2}, corresponding to a situation where $D=0.1$ and $d=1$, there holds $ -\partial _y v|_{y=0}(t=1000,\point)\approx 0$ (since the level sets of $v$ are almost perpendicular to the road), meaning that there are very few exchanges between the field and the road. This can be observed with further accuracy in Figure \ref{figure-plot-flux-D01}.
}

\refstepcounter{FIGURE}\label{figure-Solution2}
\begin{center}
\includegraphics[scale=1]{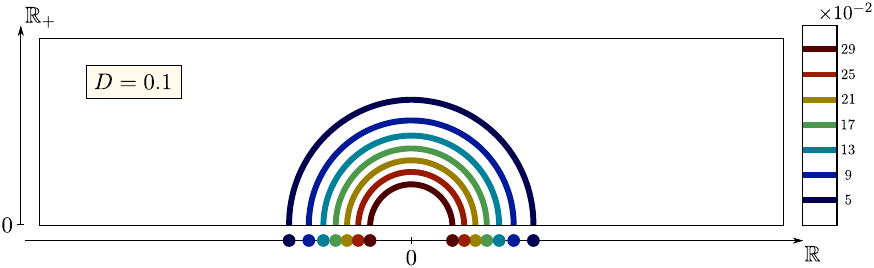}\\
\includegraphics[scale=1]{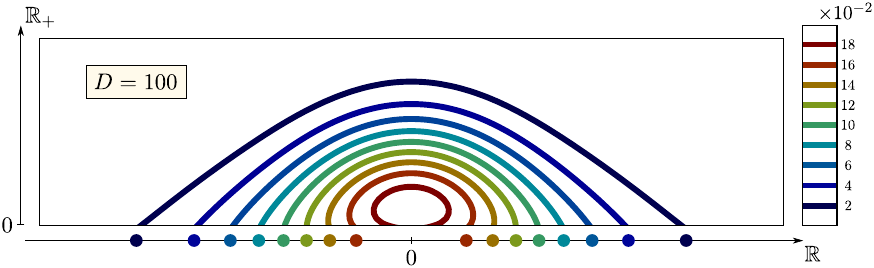}\\
\textsl{\textbf{Figure \theFIGURE} --- Two snapshots at time $t=1000$ of a few level sets of the solution to the Cauchy problem \eqref{syst}\?\?\?--\,\eqref{initial_data-bis} with $N=2$, $M=200$, $d=\mu=\nu=1$ and two different values for $D$, arising from the datum $\parentheses{v_{0},u_{0}}\equiv \parentheses{\indicatrice{\intervalleff{-5}{5}\times\intervalleff{0}{5}},0}$.}
\end{center}

\medskip

\noindent {\bf On the flux.} In view of the above considerations, we are now interested in
$$
F(t,x):= \mu u(t,x) - \nu v|_{y=0}(t,x)=-d\partial _y v|_{y=0}(t,x)
$$
which is the {\it flux entering the field}. 

When $D=100$ and $d=1$, Figure \ref{figure-plot-flux-D100} confirms that the flux is negative in the middle of the road, and positive far away. Furthermore if we denote $x_0(t)$ the (rightmost) position where the flux changes sign, that is $F(t,x_0(t))=0$, the left panel of Figure \ref{figure-plot-flux-x0-max-100} suggest that $x_0(t)$ asymptotically behaves like $\mathcal O(t^{1/2})$. Last, the right panel of Figure \ref{figure-plot-flux-x0-max-100} indicates that the asymptotic decay rate of the flux at $x=0$ is of the magnitude $\mathcal O((1+t)^{-3/2})$.

\refstepcounter{FIGURE}\label{figure-plot-flux-D100}
\begin{center}
\includegraphics[scale=1]{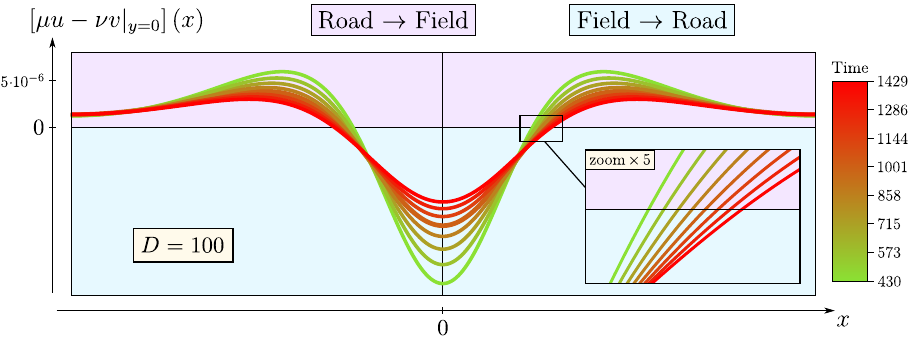}\\
\textsl{\textbf{Figure \theFIGURE} --- Plot of the flux entering the field, $F\parentheses{t,x} = \mu u(t,x) - \nu v|_{y=0}(t,x)$, of the solution to the Cauchy problem \eqref{syst}\?\?\?--\,\eqref{initial_data-bis} with $N=2$, $M=200$, $d=\mu=\nu=1$ and $D=100$, arising from the datum $\parentheses{v_{0},u_{0}}\equiv \parentheses{\indicatrice{\intervalleff{-5}{5}\times\intervalleff{0}{5}},0}$, for eight different values of time.}
\end{center}

\refstepcounter{FIGURE}\label{figure-plot-flux-x0-max-100}
\begin{center}
\includegraphics[scale=1]{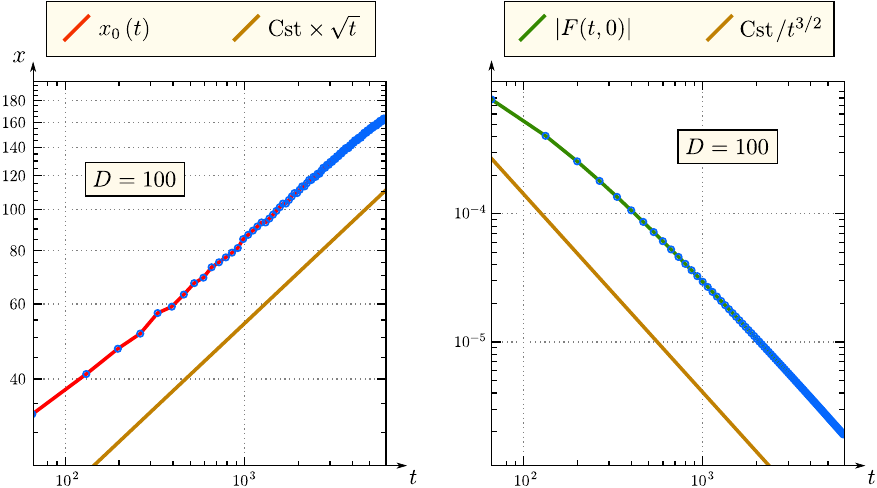}\\
\textsl{\textbf{Figure \theFIGURE} --- Plots in $\log$-scale of $x_{0}\parentheses{t}$ (on left) and
$\verti{F\parentheses{t,0}}$
(on right) for the solution to the Cauchy problem \eqref{syst}\?\?\?--\,\eqref{initial_data-bis} with $N=2$, $M=200$, $d=\mu=\nu=1$ and $D=100$, arising from the datum $\parentheses{v_{0},u_{0}}\equiv \parentheses{\indicatrice{\intervalleff{-5}{5}\times\intervalleff{0}{5}},0}$.}
\end{center}

When $D=0.1$, $d=1$, Figure \ref{figure-plot-flux-D01} indicates that the flux is positive in the middle of the road, and negative far away, that is the opposite situation than the previous case.  Furthermore if we denote $x_0(t)$ the (rightmost) position where the flux changes sign, that is $F(t,x_0(t))=0$, the left panel of Figure \ref{figure-plot-flux-x0-max-01} suggest that $x_0(t)$, again, asymptotically behaves like $\mathcal O(t^{1/2})$. Last, the right panel of Figure \ref{figure-plot-flux-x0-max-01} indicates that the asymptotic decay rate of the flux at $x=0$ is of the magnitude $\mathcal O((1+t)^{-2})$, which is in contrast with the previous case.

\refstepcounter{FIGURE}\label{figure-plot-flux-D01}
\begin{center}
\includegraphics[scale=1]{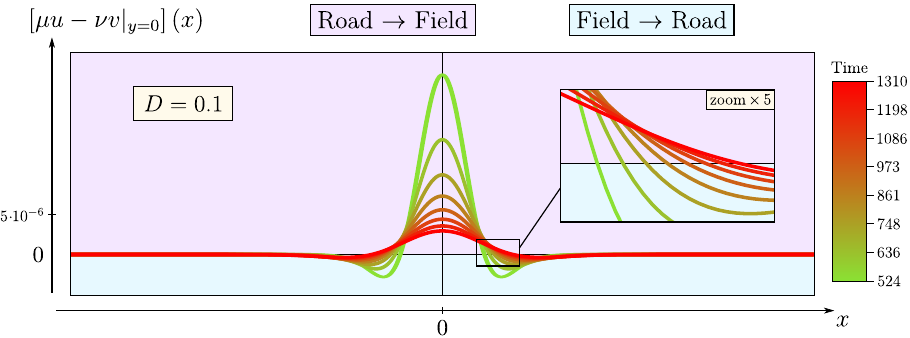}\\
\textsl{\textbf{Figure \theFIGURE} --- Plot of the flux entering the field, $F\parentheses{t,x} = \mu u(t,x) - \nu v|_{y=0}(t,x)$, of the solution to the Cauchy problem \eqref{syst}\?\?\?--\,\eqref{initial_data-bis} with $N=2$, $M=200$, $d=\mu=\nu=1$ and $D=0.1$, arising from the datum $\parentheses{v_{0},u_{0}}\equiv \parentheses{\indicatrice{\intervalleff{-5}{5}\times\intervalleff{0}{5}},0}$, for eight different values of time.}
\end{center}

\refstepcounter{FIGURE}\label{figure-plot-flux-x0-max-01}
\begin{center}
\includegraphics[scale=1]{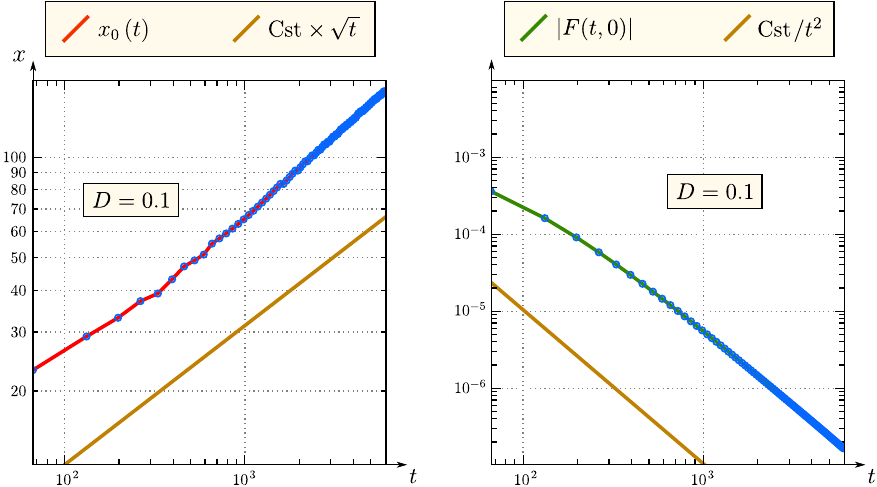}\\
\textsl{\textbf{Figure \theFIGURE} --- Plots in $\log$-scale of $x_{0}\parentheses{t}$ (on left) and
$\verti{F\parentheses{t,0}}$
(on right) for the solution to the Cauchy problem \eqref{syst}\?\?\?--\,\eqref{initial_data-bis} with $N=2$, $M=200$, $d=\mu=\nu=1$ and $D=0.1$, arising from the datum $\parentheses{v_{0},u_{0}}\equiv \parentheses{\indicatrice{\intervalleff{-5}{5}\times\intervalleff{0}{5}},0}$.}
\end{center}

\medskip

\noindent{\bf On the threshold value for the shape of the flux.} \MAT{From the above, the shape of the flux function $F$ is very dependent on the value of $D$ ($d=1$ being fixed). As we may see in
Figure \ref{figure-plot-signe-flux-fonction-de-D},
it seems there is a threshold value at $D=2$. This could be related to the enhanced spreading speed for the Fisher-KPP system \eqref{syst-non-lin}, noticed in the original paper
\cite{Ber-Roq-Ros-13-1}, when $D>2d$.
It would be interesting to have a result that highlights some connections between the purely diffusive system
\eqref{syst}
from one side, and the nonlinear system \eqref{syst-non-lin} from the other side, especially to identify more clearly the reason for this threshold.
}

\refstepcounter{FIGURE}\label{figure-plot-signe-flux-fonction-de-D}
\begin{center}
\includegraphics[scale=1]{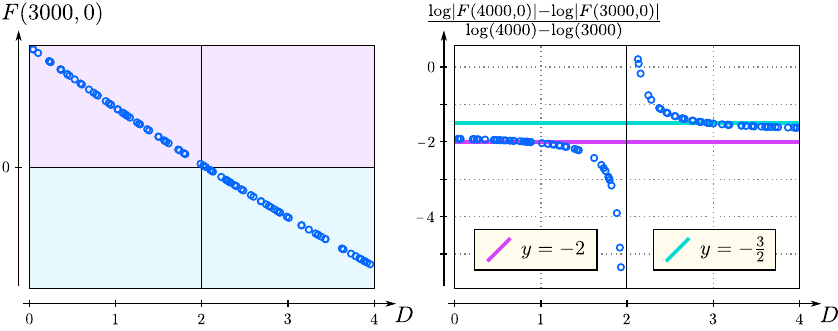}\\
\textsl{\textbf{Figure \theFIGURE} --- Plot of $F\parentheses{3000,0}$ (on left) and the decay rate of $\verti{F\parentheses{t,0}}$ (on right) with different values of $D$ for the solution to the Cauchy problem ${\text{\eqref{syst}\?\?\?--\,\eqref{initial_data-bis}}}$ with $N=2$, $M=300$, $d=\mu=\nu=1$, arising from the datum $\parentheses{v_{0},u_{0}}\equiv \parentheses{\indicatrice{\intervalleff{-20}{20}\times\intervalleff{0}{20}},0}$.}
\end{center}

Let us make a concluding remark on the asymptotic decay rate of the flux. Notice that we have focused on the flux at $x=0$. However, the $L^\infty$ norm of the flux is not always attained at this point. In particular, we have numerically observed that, in the vicinity of $x=0$, the flux is \lq\lq dromedary shaped'' (one bump) when $D$ is much smaller than $2$, but \lq\lq camel shaped'' (two bumps) when $D$ is slightly smaller than $2$. This oscillation phenomenon would deserve further investigations. 

%%%%%%%%%%%%%%%%%%%%%%%%%%%%%%%%%%%%%%%%%%%%%%%%%%%%%%%%%%%%%%%%%%%%%%%%%%
%%%%%%%%%%%%%%%%%%%%%%%%%%%%%%%%%%%%%%%%%%%%%%%%%%%%%%%%%%%%%%%%%%%%%%%%%%
%%%%%%%%%%%%%%%%%%%%%%%%%%%%%%%%%%%%%%%%%%%%%%%%%%%%%%%%%%%%%%%%%%%%%%%%%%

\appendix

\section{Appendix}

\subsection{The complementary error function $\Erfc\!$}\label{ss:appendix_ERFC}

Recalling that $\Gamma\parentheses{\ell}=e^{-\ell \, ^{2}}$, the $\Erfc\!$ function is defined by
\begin{equation}\label{def:Erfc}
\Erfc(x)\=
\frac{2}{\sqrt \pi}\int_{x}^{+\infty}\Gamma\parentheses{y} \; dy,
\qquad x\in\mathbb{R},
\end{equation}
and has an holomorphic continuation, namely
\begin{equation}\label{def:Erfc_complex}
\Erfc(z)\=
1-
\frac{2}{\sqrt{\pi}}
\sum\limits_{k=0}^{+\infty}\frac{\parentheses{-1}^{k}}{\parentheses{2k+1}k!}z^{2k+1},
\qquad z\in\mathbb{C}.
\end{equation}

For any small $\delta >0$, we have the asymptotic expansion, see \cite[page 393]{Hen-book} or \cite[page 262]{Cuy-book},
\begin{equation}\label{DAS:Erfc}
\sqrt{\pi} \, \frac{\Erfc}{\Gamma}\parentheses{z}
=
\frac{1}{z}
-\frac{1}{2z^{3}}
+\frac{3}{4z^{5}}
+o\Parentheses{\frac{1}{\verti{z}^6}},
\qquad
\text{as }
\verti{z}\to +\infty, 
\end{equation}
uniformly in $\{z\in \mathbb C:\verti{\text{arg}\parentheses{z}}\leq \frac{3\pi}{4}-\delta\}$ --- see Figure \ref{figure-ERFC-on-GAMMA}. 
By computing the derivatives of $\frac{\Erfc}{\Gamma}$ and using \eqref{DAS:Erfc} one can check that, for all non-zero $z\in\mathbb{R}_{+}+i\mathbb{R}$,
\begin{equation}\label{Control_ERFC}
\verti{\frac{\Erfc}{\Gamma}\parentheses{z}} \lesssim \frac{1}{\verti{z}},\qquad\qquad
\verti{\frac{d}{dz}\frac{\Erfc}{\Gamma}\parentheses{z}} \lesssim \frac{1}{\verti{z}^{2}},
\qquad\qquad
\verti{\frac{d^{\, 2}}{dz^{2}}\frac{\Erfc}{\Gamma}\parentheses{z}} \lesssim \frac{1}{\verti{z}^{3}}.
\end{equation}

\begin{center}
\begin{tabular}{ccc}
\includegraphics[scale=1]{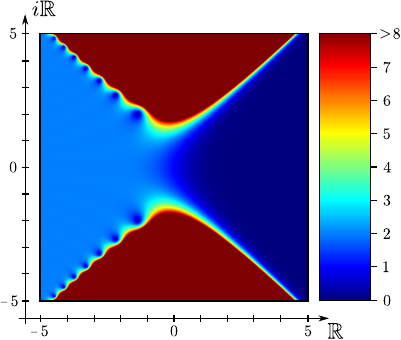}
&\hspace*{3.083mm}&
\includegraphics[scale=1]{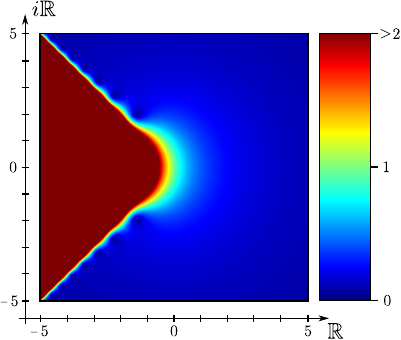}\\[1mm]
\refstepcounter{FIGURE}\label{figure-ERFC}
\textsl{\textbf{Figure \theFIGURE} --- Plot of $\verti{\Erfc\parentheses{z}}$.}
&&
\refstepcounter{FIGURE}\label{figure-ERFC-on-GAMMA}
\textsl{\textbf{Figure \theFIGURE} --- Plot of $\verti{\frac{\Erfc}{\Gamma}\parentheses{z}}$.}
\end{tabular}
\end{center}

\subsection{The polynomials $P_{\delta}$}\label{ss:appendix_P_delta}

Let $\mu , \nu , d>0$ be given. For $\delta \geq 0$, we consider the polynomial 
$$
P_{\delta}\parentheses{\sigma} :=
\sigma^{3} + A\sqrt{d} \sigma^{2} + \parentheses{\mu + \delta} \, \sigma + A\sqrt{d}\delta,
$$
where $A \= \frac{\nu}{d}$. We denote $\alpha\parentheses{\delta}$, $\beta\parentheses{\delta}$ and $\gamma\parentheses{\delta}$ its three complex roots.  Notice that $0$ is a root if and only if $\delta=0$. Recall that $P_\delta$ has multiple roots if and only if the discriminant
$$
\Delta_{\delta}:=
18A^{2}d\parentheses{\mu+\delta}\delta
- 4A^{4}d^{2}\delta
+ A^{2}d\parentheses{\mu+\delta}^{2}
- 4\parentheses{\mu+\delta}^{3}
- 27 A^{2}d\delta^{2}
$$
vanishes. As a consequence, $\alpha$, $\beta$ and $\gamma$ are distinct for almost all $\delta \geq 0$. 

As easily checked, the real roots belong to $(-A\sqrt d,0]$ and, since $\alpha+\beta+\gamma=-A\sqrt d$, we  have
\begin{equation}\label{partie-reelle-neg}
-A\sqrt{d} \; < \; \realpart{\alpha\parentheses{\delta}} , \; \realpart{\beta\parentheses{\delta}} , \; \realpart{\gamma\parentheses{\delta}} \; \leq \; 0,
\qquad
\text{for all }\delta\geq 0.
\end{equation}
Then, a more detailed analysis provides further information about the nature and the behavior of the roots, as stated in the following lemmas.

%LEMME P_{\delta} (large \delta)
\begin{lemma}[Large values of $\delta$]\label{lem:fact_P_delta_large_values_of_delta}
There is  $\delta_{\infty} > 0$ large enough such that, for all $\delta\geq \delta_{\infty}$, the three roots are simple, one is real (say $\alpha$) and the two others are complex conjugated (say $\impart{\beta}>0$ and $\impart{\gamma}<0$). Moreover, we have
$$
\lim\limits_{\delta \rightarrow +\infty} \alpha\parentheses{\delta} = -A\sqrt{d},
\qquad
\lim\limits_{\delta \rightarrow +\infty} \realpart{\beta\parentheses{\delta}} = 0,
\qquad
\lim\limits_{\delta \rightarrow +\infty} \impart{\beta\parentheses{\delta}} = +\infty,
$$
so that, up to increasing $\delta_\infty$ if necessary,  $\verti{\beta-\gamma}>\verti{\beta}$ for all $\delta \geq  \delta_{\infty}$.
\end{lemma}

Next we fix $\nu>0$ and $d>0$ (and thus $A$) and describe the situations while decreasing the parameter $\mu$ from $+\infty$ to $0$. Most of the results can be seen at a glance on  Figure \ref{figure-Roots}. 

\refstepcounter{FIGURE}\label{figure-Roots}
\begin{center}
\includegraphics[scale=1]{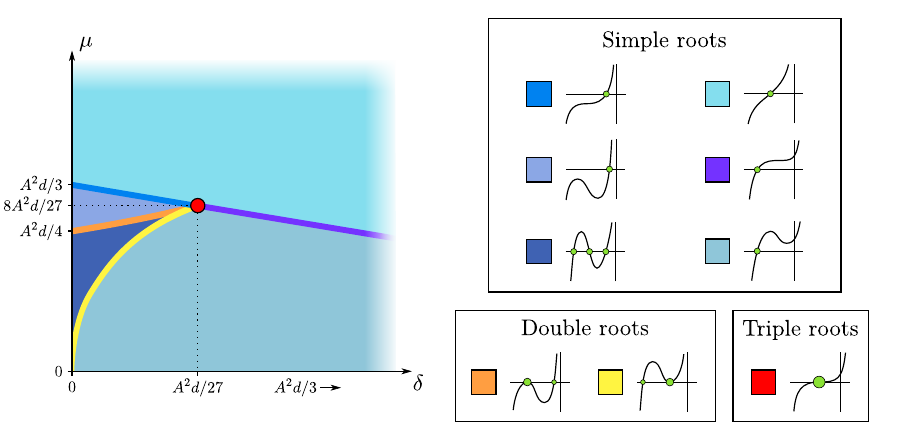}\\\textsl{\textbf{Figure \theFIGURE} --- Multiplicity of the roots of $P_{\delta}$ for fixed $\nu$, $d$. Fix the last parameter $\mu>0$, then $\delta$ browses $[0,+\infty)$.}
\end{center}

%LEMME P_{\delta} (simple roots)
\begin{lemma}[Only simple roots]\label{lem:fact_P_delta_only_simple_roots}
Assume $\mu>\frac{8A^{2}d}{27}$ is fixed and set ${\Eb} \= \mathbb{R}_{+}$.
\begin{itemize}
\refstepcounter{ASimpleRoots}\label{ASimpleRoots-kind-of-roots}
	\item [$(\theASimpleRoots)$] For all $\delta\in{\Eb}$, the three roots are simple, one is real and the two others are complex conjugated.
\refstepcounter{ASimpleRoots}\label{ASimpleRoots-roots-staying-away}
	\item [$(\theASimpleRoots)$] There is $\varepsilon>0$ such that $\inf\limits_{\delta\in{\Eb}} \Parentheses{\verti{\alpha-\beta} , \verti{\alpha-\gamma} , \verti{\beta-\gamma}} > \varepsilon$.
\end{itemize}
\end{lemma}

%LEMME P_{\delta} (triple root)
\begin{lemma}[Appearance of a triple root]\label{lem:fact_P_delta_triple-root}
Assume $\mu=\frac{8A^{2}d}{27}$ and set $\delta_{0} \= \frac{A^{2}d}{27}$. 
\begin{itemize}
\refstepcounter{ATripleRoots}\label{ATripleRoots-kind-of-roots}
	\item [$(\theATripleRoots)$] If $\delta \neq \delta_{0}$, the three roots are simple, one is real and the two others are complex conjugated, and as $\delta$ encounters $\delta_{0}$, the three simple roots merge in a triple root $\lambda_{0} = -\frac{A\sqrt{d}}{3}$.
\end{itemize}
Without loss of generality, we assume that $\alpha$ is the real root for all $\delta\geq 0$.
\begin{itemize}
	\refstepcounter{ATripleRoots}\label{ATripleRoots-epsR-espI}
	\item [$(\theATripleRoots)$] Letting $\epsR \= \realpart{\beta} - \lambda_{0}$ and $\epsI \= \impart{\beta}$, one has
$$
\alpha = \lambda_{0} - 2\epsR,
\qquad
\beta = \lambda_{0} + \epsR + i\epsI,
\qquad
\gamma = \lambda_{0} + \epsR - i\epsI,
$$
and $\epsR \stackrel{\delta\to\delta_{0}}{=}\mathcal{O}\parentheses{\epsI}$.
\end{itemize}
Define $\eta \= \frac{\delta_{0}}{2}$ and ${\Eb} \= \mathbb{R}_{+}\setminus\intervalleoo{\delta_{0}-\eta}{\delta_{0}+\eta}$, then there is $\varepsilon>0$ such that
\begin{itemize}
\refstepcounter{ATripleRoots}\label{ATripleRoots-roots-staying-away}
	\item [$(\theATripleRoots)$] $\inf\limits_{\substack{\delta\in{\Eb}}} \Parentheses{\verti{\alpha-\beta} , \verti{\alpha-\gamma} , \verti{\beta-\gamma}} > \varepsilon$,
\refstepcounter{ATripleRoots}\label{ATripleRoots-seg_away_from_zero}
	\item [$(\theATripleRoots)$] $\inf\limits_{\substack{\delta\in{\Eb}^{c}}}\;\;\inf\limits_{\substack{\lambda\in\intervalleff{\alpha}{\beta}}} \;\;\verti{\lambda} > \varepsilon$.
\end{itemize}
\end{lemma}

%LEMME P_{\delta} (two double roots)
\begin{lemma}[Appearance of a double root, two times]\label{lem:fact_P_delta_2-double-roots}
Assume $\frac{A^{2}d}{4}\leq\mu<\frac{8A^{2}d}{27}$, then there are $0\leq\delta_{1}<\delta_{2}$ such that the following holds.
\begin{itemize}
\refstepcounter{ATwoDoubleRoots}\label{ATwoDoubleRoots-kind-of-roots}
	\item [$(\theATwoDoubleRoots)$] If $\delta \neq \delta_{1}$ and $\delta \neq \delta_{2}$, the three roots are simple and as $\delta$ encounters $\delta_{i}$ $\parentheses{i=1\text{ or }2}$, two of the three simple roots merge in a double root $\lambda_{i}$.
\end{itemize}
Without loss of generality, we assume that, for $i=1$ or $2$, the double root $\lambda_{i}$ derives from the merging of $\alpha$ and $\beta$. Define $\eta \= \frac{\delta_{2}-\delta_{1}}{2}$ and ${\Eb} \= \mathbb{R}_{+}\setminus\crochets{\intervalleoo{\delta_{1}-\eta}{\delta_{1}+\eta}\cup\intervalleoo{\delta_{2}-\eta}{\delta_{2}+\eta}}$, then there is $\varepsilon>0$ such that
\begin{itemize}
\refstepcounter{ATwoDoubleRoots}\label{ATwoDoubleRoots-roots-staying-away}
	\item [$(\theATwoDoubleRoots)$] $\inf\limits_{\substack{\delta\in{\Eb}}} \Parentheses{\verti{\alpha-\beta} , \verti{\alpha-\gamma} , \verti{\beta-\gamma}} > \varepsilon$,
\refstepcounter{ATwoDoubleRoots}\label{ATwoDoubleRoots-seg-away-from-zero-and-gamma}
	\item [$(\theATwoDoubleRoots)$] $\inf\limits_{\substack{\delta\in{\Eb}^{c}}}\;\;\inf\limits_{\substack{\lambda\in\intervalleff{\alpha}{\beta}}}\;\; \Parentheses{\verti{\lambda} , \verti{\lambda-\gamma}} > \varepsilon$.\end{itemize}
\end{lemma}

%LEMME P_{\delta} (one double root)
\begin{lemma}[Appearance of a double root, one time]\label{lem:fact_P_delta_1-double-root}
Assume $0<\mu<\frac{A^{2}d}{4}$, then there is $\delta_{2}>0$ such that the following holds.
\begin{itemize}
\refstepcounter{AOneDoubleRoots}\label{AOneDoubleRoots-kind-of-roots}
	\item [$(\theAOneDoubleRoots)$] If $\delta \neq \delta_{2}$, the three roots are simple and as $\delta$ encounters $\delta_{2}$, two of the three simple roots merge in a double root $\lambda_{2}$.
\end{itemize}
Without loss of generality, we assume that the double root $\lambda_{2}$ derives from the merging of $\alpha$ and $\beta$. Define $\eta \= \frac{\delta_{2}}{2}$ and ${\Eb} \= \mathbb{R}_{+}\setminus\intervalleoo{\delta_{2}-\eta}{\delta_{2}+\eta}$, then there is $\varepsilon>0$ such that
\begin{itemize}
\refstepcounter{AOneDoubleRoots}\label{AOneDoubleRoots-roots-staying-away}
	\item [$(\theAOneDoubleRoots)$] $\inf\limits_{\substack{\delta\in{\Eb}}} \Parentheses{\verti{\alpha-\beta} , \verti{\alpha-\gamma} , \verti{\beta-\gamma}} > \varepsilon$,
\refstepcounter{AOneDoubleRoots}\label{AOneDoubleRoots-seg-away-from-zero-and-gamma}
	\item [$(\theAOneDoubleRoots)$] $\inf\limits_{\substack{\delta\in{\Eb}^{c}}}\;\;\inf\limits_{\substack{\lambda\in\intervalleff{\alpha}{\beta}}} \;\;\Parentheses{\verti{\lambda} , \verti{\lambda-\gamma}} > \varepsilon$.\end{itemize}
\end{lemma}

\subsection{Fourier, Laplace and Fourier/Laplace transforms}\label{ss:appendix_Transforms}

The transforms of a function $w$ depending on $t>0$ and $x\in \R^{N-1}$ we use are
\begin{itemize}
	\item the $x$-Fourier one (said \lq\lq \textit{hat} $w$''):
	$$
	\hat{w}\parentheses{t,\xi}
	=
	{\Fb}\crochets{w\parentheses{t,\point}}\parentheses{\xi}
	\=
	\int_{\mathbb{R}^{N-1}}^{}
		w\parentheses{t,z}e^{-i \xi \cdot z}
	dz;
	$$
	\item the $t$-Laplace one (said \lq\lq \textit{frown} $w$''):
	$$
	\frown{w}\parentheses{s,x}
	=
	{\Lb}\crochets{w\parentheses{\point,x}}\parentheses{s}
	\=
	\int_{0}^{+\infty}
		w\parentheses{\tau,x}e^{-s\tau}
	d\tau;
	$$
	\item the $x$-Fourier/$t$-Laplace one (said \lq\lq \textit{frown-hat} $w$''):
	$$
	\frownhat{w}\parentheses{s,\xi}
	=
	{\Fb}{\Lb}\crochets{w\parentheses{\point,\point}}\parentheses{s,\xi}
	\=
	\int_{\mathbb{R}^{N-1}}^{}
		\int_{0}^{+\infty}
			w\parentheses{\tau,z}e^{-\parentheses{s\tau + i \xi \cdot z}}
		d\tau
	dz.
	$$
\end{itemize}

The following Fourier and Laplace transforms are well-known  --- see, e.g., \cite[page 250, lines 84 and 88]{Spi-book} for $(\ref{A1-Laplace-E_sur_sig})$ and $(\ref{A1-Laplace-E_sur_sig_fois_sig_plus_b})$.

%LEMME A1
\begin{lemma}[Useful $x$-Fourier and $t$-Laplace transforms]\label{lem:inverse-Laplace-Fourier}
For all $a\in\mathbb{R}$ and all $b\in\mathbb{C}$,
\begin{itemize}
\refstepcounter{AUn}\label{A1-Fourier-Gaussienne}
	\item [$(\theAUn)$] ${\Fb}\crochets{x\mapsto \frac{e^{-\frac{\vertii{x}{}^{2}}{4a}}}{\parentheses{4\pi a}^{\frac{N-1}{2}}}}\parentheses{\xi} = e^{-a\vertii{\xi}{}^{2}}$,
	\qquad $a>0$.
\refstepcounter{AUn}\label{A1-Laplace-un-sur-a-plus-s}
	\item [$(\theAUn)$] ${\Lb}\crochets{t\mapsto e^{-at}}\parentheses{s} = \frac{1}{a + s}$,
	\qquad $s>-a$.
\refstepcounter{AUn}\label{A1-Laplace-E_sur_sig}
	\item [$(\theAUn)$] ${\Lb}\crochets{t\mapsto\frac{e^{-\frac{a^{2}}{4t}}}{\sqrt{\pi t}}}\parentheses{s} = \frac{e^{-a\sqrt{s}}}{\sqrt{s}}$,
	\qquad $s>0$.
\refstepcounter{AUn}\label{A1-Laplace-E_sur_sig_fois_sig_plus_b}
	\item [$(\theAUn)$] ${\Lb}\crochets{t\mapsto\frac{\Erfc}{\Gamma}\Parentheses{\frac{a+2bt}{2\sqrt{t}}}e^{-\frac{a^{2}}{4t}}}\parentheses{s} = \frac{e^{-a\sqrt{s}}}{\sqrt{s}\Parentheses{\sqrt{s}+ \, b}}$,
	\qquad $s>\realpart{b}^{2}$.
\refstepcounter{AUn}\label{A1-Laplace-E_sur_sig_moins_b}
	\item [$(\theAUn)$] ${\Lb}\crochets{t\mapsto e^{-\frac{a^{2}}{4t}} \Parentheses{\frac{1}{\sqrt{\pi t}}+ b \, \frac{\Erfc}{\Gamma}\Parentheses{\frac{a-2bt}{2\sqrt{t}}}}}\parentheses{s} = \frac{e^{-a\sqrt{s}}}{\sqrt{s}- \, b}$,
	\qquad $s>\realpart{b}^{2}$.
\end{itemize}
where $\Gamma\parentheses{\ell} \= e^{-\ell^{\, 2}}$ and the definition of $\Erfc\!$ being recalled in Appendix \ref{ss:appendix_ERFC}.\end{lemma}

\begin{lemma}[Properties on $x$-Fourier and $t$-Laplace transforms]\label{lem:propri-Laplace-Fourier}
For $f,g : \mathbb{R}\to\mathbb{C}$, all $a\in\mathbb{R}$ and $s>0$,
\begin{itemize}
\refstepcounter{ADeux}\label{A2-Fourier-inversion}
	\item [$(\theADeux)$] $f\parentheses{x} = \frac{1}{\parentheses{2\pi}^{N-1}} \int_{\mathbb{R}^{N-1}}^{}{\Fb}\crochets{f}\parentheses{\xi} \, e^{i \xi\cdot x} d\xi$,
\refstepcounter{ADeux}\label{A2-Fourier-convolee}
	\item [$(\theADeux)$] ${\Fb}\crochets{f}\times{\Fb}\crochets{g} = {\Fb}\crochets{f\ast g} = {\Fb}\crochets{x\mapsto \int_{\mathbb{R}^{N-1}}^{}f\parentheses{x - z}g\parentheses{z}dz}$,
\refstepcounter{ADeux}\label{A2-Laplace-convolee}
	\item [$(\theADeux)$] ${\Lb}\crochets{f}\times{\Lb}\crochets{g} = {\Lb}\crochets{t\mapsto \int_{0}^{t}f\parentheses{t - \tau}g\parentheses{\tau}d\tau}$,
\refstepcounter{ADeux}\label{A2-Delay_th}
	\item [$(\theADeux)$] ${\Lb}\crochets{t\mapsto e^{-at}f(t)}\parentheses{s}={\Lb}\crochets{f}\parentheses{s+a}$.
\end{itemize}
\end{lemma}

\medskip

\noindent{\bf Acknowledgement.} Matthieu Alfaro is supported by  the {\it région Normandie} project BIOMA-NORMAN 21E04343 and the ANR project DEEV ANR-20-CE40-0011-01. Samuel Tréton would like to acknowledge the {\it région Normandie} for the financial support of his PhD.
\MAT{The authors are grateful to the anonymous referee whose precise comments have improved the presentation of the results.} 

\bibliographystyle{siam}  

\bibliography{biblio}

\end{document}